\documentclass[11pt]{article}
\usepackage{amsmath,amsthm,amssymb}
\usepackage{amsfonts}
\usepackage[dvips]{graphicx}        
\usepackage{multicol}        
\usepackage[bottom]{footmisc}
\usepackage{colortbl}
\usepackage{float}
\usepackage{hhline,multirow}
\usepackage{graphicx}
\usepackage{algorithm}
\usepackage{algpseudocode}

\usepackage[left=1in,right=1in,
    top=1in,bottom=1in]{geometry}

\newtheorem{theorem}{Theorem}[section]
\newtheorem{corollary}[theorem]{Corollary}
\newtheorem{lemma}[theorem]{Lemma}
\newtheorem{definition}[theorem]{Definition}

\newtheoremstyle{normalnoit}{}{}{}{ }{\bf }{.}{ }{}
\theoremstyle{normalnoit}
\newtheorem{example}[theorem]{Example}
\newtheorem{remark}[theorem]{Remark}

\def\invddots{\mathinner{\mskip1mu\raise1pt\vbox{\kern7pt\hbox{.}}\mskip2mu
        \raise4pt\hbox{.}\mskip2mu\raise7pt\hbox{.}\mskip1mu}}
\def\dddots{\mathinner{\mskip1mu\raise4pt\vbox{\kern7pt\hbox{.}}\mskip2mu
        \raise3pt\hbox{.}\mskip2mu\raise1pt\hbox{.}\mskip1mu}}

\newfont{\ffa}{cmbx12}

\newcommand{\be}{\begin{equation}}
\newcommand{\ee}{\end{equation}}
\newcommand{\bs}{\begin{subequations}}
\newcommand{\es}{\end{subequations}}
\newcommand{\bt}{\begin{theorem}}
\newcommand{\et}{\end{theorem}}
\newcommand{\bd}{\begin{definition}}
\newcommand{\ed}{\end{definition}}
\newcommand{\ba}{\left[ \begin{array}}
\newcommand{\ea}{\\ \end{array} \right]}

\newcommand{\diag}{\mathrm{diag}}

\author{Sirani M. Perera, Grigory Bonik and Vadim Olshevsky}
\title{\textbf{A Fast Algorithm for the Inversion of Quasiseparable Vandermonde-like Matrices}}
\date{}

\begin{document}


\maketitle

\begin{abstract}
The results on Vandermonde-like matrices were introduced as a generalization of polynomial Vandermonde matrices, and the displacement structure of these matrices was used to derive an inversion formula. In this paper we first present a fast Gaussian elimination algorithm for the polynomial Vandermonde-like matrices. Later we use the said algorithm to derive fast inversion algorithms for quasiseparable, semiseparable and well-free Vandermonde-like matrices having $\mathcal{O}(n^2)$ complexity. To do so we identify structures of displacement operators  in terms of generators and the recurrence relations(2-term and 3-term) between the columns of the basis transformation matrices for quasiseparable, semiseparable and well-free polynomials. Finally we present an $\mathcal{O}(n^2)$ algorithm to compute the inversion of quasiseparable Vandermonde-like matrices.
\end{abstract}


\section{Introduction}
Structure generalization of the classical Vandermonde matrix $V(x)=[x_i^{j-1}]$ is the polynomial Vandermonde matrix of the form 
\be 
  V_Q(x) = \left[
        \begin{array}{cccc}
            Q_0(x_1) & Q_1(x_1) & \cdots & Q_{n-1}(x_1) \\
            Q_0(x_2) & Q_1(x_2) & \cdots & Q_{n-1}(x_2) \\
            \vdots   & \vdots   &        & \vdots       \\
            Q_0(x_n) & Q_1(x_n) & \cdots & Q_{n-1}(x_n)
        \end{array}
        \right]
\label{eq:vq}
\ee
where $x=[x_1,x_2,\cdots,x_n]$ and the set of polynomials $Q=\{Q_0(x), Q_1(x), \cdots, Q_{n-1}(x)\}$ satisfies deg $Q_k(x)=k$. Like classical and polynomial Vandermonde matrices, Toeplitz $[t_{i-j}]$, Hankel $[h_{i+j-2}]$, Toeplitz-plus-Hankel, Cauchy $\left[\frac{1}{x_i-y_j} \right]$, Pick matrices, etc. have structures. Once the structure is considered these matrices can be used in many applications, e.g. solving systems of questions, calculating Gaussian quadrature, theories in interpolations and approximations. All of the above structured matrices have low displacement rank, and hence these matrices are said to be $\textit{like}$ matrices in displacement structure theory. This low displacement rank property was first introduced by \cite{KKM79} for Toeplitz matrices and later it was recognized by \cite{HR84} that this property is common for all the other mentioned matrices as well.    
\newline\newline
The Structure-ignoring approach of Gaussian elimination for the inversion of polynomial Vandermonde matrices $V_Q(x)$  costs $\mathcal{O}(n^3)$ operations. Once the structure of $V_Q(x)$ or recurrence relations of the polynomial system $\{Q\}$ is considered, the resulting algorithm is cheaper and costs only $\mathcal{O}(n^2)$ operations. Table \ref{tbl:oldwrk} shows the previous work in deriving such fast inversion formulas, inversion algorithms, and algorithms for solving linear systems corresponding to the class of polynomial Vandermonde matrices.    
\begin{table}[ht]
\begin{center}
\begin{tabular}{|l|l|l|l|l|}
\hline
Vandermonde & Polynomial System $Q$    & $\mathcal{O}(n^2)$  &  $\mathcal{O}(n^2)$   & $\mathcal{O}(n^2)$  \\ 
     Matrix $V_Q(x)$                   &                       & inversion & inversion & system \\

                        &                       & formula & algorithm & solver \\
\hline
        Classical-V                      &   monomials  &   P \cite{P64}, Tr \cite{T66}, GO \cite{GO97}  &  P \cite{P64}, Tr\cite{T66} & BP\cite{BP70} \\ 
\hline        
        Chebychev-V                      & Chebychev poly           & GO \cite{GO94}  &GO \cite{GO94} & RO\cite{RO91}\\ 
\hline        
        Three-Term-V                     & Real orthogonal poly           & Vs \cite{V88}, GO \cite{GO94} & CR \cite{CR93} & Hi\cite{H90}\\ 
\hline        
        Szeg\"o--V & Szeg\"o polynomial  & O \cite{O98} & O \cite{O01} & BEGKO \cite{BEGKO07} \\ 
\hline        
        Quasiseparable                  & Quasiseparable  & BEGOT \cite{BEGOT} & BEGOT \cite{BEGOT12},& BEGKO \cite{BEGKO09}\\ 
          Vandermonde        & polynomial & BEGOT \cite{BEGOT13} & BEGOTZ \cite{BEGOTZ10} & \\ 
\hline
\end{tabular}
\end{center}
\caption{Fast $\mathcal{O}(n^2)$ inversion for polynomial--Vandermonde
matrices.}\label{tbl:oldwrk}
\end{table}
\newline\newline
Inversion formulas and fast system solving are classical applications in Displacement Theory.  Thus it was natural to derive formulas and algorithms not only for polynomial Vandermonde matrices but also for polynomial Vandermonde-like matrices. These inversion and system solving results for {\it like}-matrices can be found in see e.g. \cite{HR84}, \cite{GO942}, \cite{GO97}, \cite{KO95}, \cite{KO97}, \cite{KS95} mentioned quite a few.  
\subsection{Displacement equations and Vandermonde-like matrices}
The low rank displacement property or ``{\it{like}}" idea allowed one to nicely unify and extend the results of polynomial Vandermonde matrices to polynomial Vandermonde-like matrices while preserving displacement structure under inversion. 
\newline\newline
Let's start this section with the definition of displacement equation and the rank of the displacement operator.
\bd
A linear displacement operator $\Theta_{\Omega ,M ,F,N}(.) : C^{n\times n}   \rightarrow  C^{n\times n}$ is a function which transforms each matrix $R \in C^{n\times n}$ to its displacement equation 
\be
\Theta_{\Omega ,M ,F,N}(R)=\Omega \:R\:M -F\:R\:N=G\:B 
\label{eq:dko}
\ee
where $\Omega ,M ,F,N \in C^{n\times n}$ are given matrices and $G \in C^{n\times \alpha},\:B \in C^{\alpha\times n}$. The pair $\{G, B\}$ on last right in (\ref{eq:dko}) is called a minimal generator of $R$ and
\be
rank \:\left\{ \Theta_{\Omega ,M ,F,N}(R) \right \}=\alpha. 
\label{rnk}
\ee
\ed
\begin{example}
Toeplitz matrix $T=[t_{i-j}]_{1 \leq i, j \leq n}$ satisfies the displacement equation
\be
\begin{aligned}
 T-Z_0 \cdot T \cdot Z_0^T & = \begin{bmatrix}
t_0 &t_{-1}  &\cdots  &t_{-n+1} \\ 
t_1 &0  & \cdots &0 \\ 
 \vdots&\vdots  &\ddots  & \vdots\\ 
 t_{n-1}&0  &\cdots  & 0
\end{bmatrix}
&\\=& \left[\begin{array}{cc}
\frac{t_0}{2} & 1\\ 
 t_1& 0\\ 
 \vdots& \vdots\\ 
 t_{n-1}&0 
\end{array} \right] \left[\begin{array}{llll}
1 &0  & \cdots & 0\\ 
\frac{t_0}{2} &t_{-1}  &\cdots  &t_{-n+1} 
\end{array}\right]
\end{aligned}
\label{tde}
\ee
where $Z_0$ is a lower shift matrix. Following (\ref{tde}), $\textrm{rank}\left \{ \Theta_{I, I, Z_0, Z_0^T}(T) \right \} = 2$.
\end{example}
\begin{example} 
Hankel matrix $H=[h_{i+j-2}]_{1 \leq i,j \leq n}$ satisfies the displacement equation
\be
\begin{aligned}
Z_0 \cdot H-H \cdot Z_0^T & =  \left[
\begin{array}{ccccc} h_0 & h_1 & h_2 & \cdots & h_{n-1}\\ 
h_1 & h_2 & h_3 & \cdots & h_n \\ \vdots &  &  &  &\vdots\\
h_{n-1} & h_n & h_{n+1} & \cdots & h_{2n-2} 
\end{array} \right] \nonumber \\ 
& =  \left[
\begin{array}{cc} 1 & 0\\ 
0 & h_0 \\ \vdots & \vdots\\ 0 & h_{n-2} 
\end{array} \right]
\left[
\begin{array}{rrrr}
0 & -h_0 & \cdots & -h_{n-2}\\ 
1 & 0 & \cdots & 0 \end{array} \right]
\end{aligned}
\label{hde}
\ee
where $Z_0$ is a lower shift matrix. Following (\ref{hde}), $\textrm{rank}\left \{ \Theta_{Z_0, I, I, Z_0^T}(H) \right \} = 2$.
\end{example}
\noindent Next we introduce a displacement operator of the polynomial Vandermonde-like matrices using a recurrence relation of the polynomials. 
\subsubsection{Displacement operator based on recurrence relations}
Let $Q = \{Q_0(x), Q_1(x), \ldots, Q_{n-1}(x)\}$ with $\deg Q_k = k$ be a system of $n$ polynomials satisfying recurrence relations:
\begin{align}
Q_0(x) &= \tau_0 \notag \\
Q_k(x) &= \tau_k \cdot x \cdot Q_{k-1}(x) - a_{k-1,\:k}Q_{k-1}(x) - a_{k-2,\:k}Q_{k-2}(x) - \cdots - a_{0,\:k}Q_{0}(x) 
\label{generalrecur}
\end{align}
for some coefficients $\{\tau_k\}$ and $\{a_{jk}\}$. Given recurrence relations (\ref{generalrecur}), we can define two
upper triangular matrices
\begin{equation}
M_Q=\left[
    \begin{array}{ccccc}
           1 & a_{0,1} & a_{0,2}   & \cdots &   a_{0,n-1} \\
           0 &       1 & a_{1,2} & \ldots &   a_{1,n-1} \\
      \vdots &  \ddots &  \ddots &   \ddots     &      \vdots \\
      \vdots &         &  \ddots & \ddots & a_{n-2,n-1} \\
           0 &  \cdots &  \cdots &      0 &           1
    \end{array}
\right],
\quad
N_Q=\left[
    \begin{array}{ccccc}
         0 & \tau_1 &        0 & \cdots &            0 \\
         0 &        0 & \tau_2 & \ddots &            0 \\
         0 &        0 &        0 & \ddots &       \vdots \\
    \vdots &          &          & \ddots & \tau_{n-1} \\
         0 &   \cdots &   \cdots &      0 &            0
    \end{array}
\right].
\label{eq:mn}
\end{equation}
Following \cite{KO97} one can show that the polynomial Vandermonde matrix $V_Q(x)$ satisfies the \emph{displacement equation}:
\begin{equation} \label{disp1v}
V_Q(x) \cdot M_Q - D_x \cdot V_Q(x) \cdot N_Q =
\left[\begin{array}{c}
1 \\
1 \\
\vdots \\
1 \\
 \end{array}\right]
\cdot
\left[\begin{array}{cccc}
\tau_0 &
0 &
\cdots &
0
 \end{array}\right],
\end{equation}
where $D_x = \diag(x_1,x_2,\ldots,x_n)$. Note that the matrix on the right-hand side of equation (\ref{disp1v})
is of rank one. Thus when the system of polynomials $\{Q\}$ satisfies the recurrence relations (\ref{generalrecur}), the class of polynomial Vandermonde matrices can be generalized by allowing the displacement rank to be one. We say that matrix $R$ is \emph{Vandermonde-like} if it satisfies the displacement equation:
\begin{equation} \label{disp1}
R \cdot M_Q - D_x \cdot R \cdot N_Q = \tilde{G}\cdot \tilde{B},
\end{equation}
where $\tilde{G} \in \mathbb{C}^{n \times \alpha}$,
      $\tilde{B} \in \mathbb{C}^{\alpha \times n}$
and $\alpha$ is the \emph{displacement rank} of matrix $R$ which is small compared to $n$. By \cite{KO97}, one can see that the matrix $R$ is uniquely identified by its displacement equation. 
\newline\newline
Throughout our discussion we will mostly use another form of a displacement operator $W_Q$ which is stated as follows.
\newline\newline
If all $x_1,x_2,\ldots,x_n$ are nonzero, then $D_x$ is invertible and
$D_x^{-1} = D_{\frac{1}{x}} = \diag \left (\frac{1}{x_1},\frac{1}{x_1},\ldots,\frac{1}{x_n} \right)$.
Multiplying equation (\ref{disp1}) by $D_\frac{1}{x}$ from the left
and $M_Q^{-1}$ from the right, we get
\begin{equation}
\label{eq:dispwq}
D_{\frac{1}{x}} \cdot R - R \cdot W_Q = G \cdot B,
\end{equation}
where $G = D_\frac{1}{x}\tilde{G}$, $B = \tilde{B} M_Q^{-1}$, and
\begin{equation} 
\label{eq:wq} 
W_Q = N_Q M_Q^{-1}. 
\end{equation}
\subsection{Generalized associated polynomials}
The Traub algorithm \cite{P64,T66} computes the entries of $V_Q(x)^{-1}$ in $\mathcal{O}(n^2)$ operations. This algorithm was derived by using the properties of ${\it associated \:polynomials \:(Horner\: polynomials)}$. These generalized associated polynomials can be defined as follows.
\newline\newline
Let $Q = \{Q_0(x), \ldots, Q_{n-1}(x)\}$ be a system of polynomials
satisfying recurrence relations (\ref{generalrecur}).
We define a system $\widehat{Q}=\{\widehat{Q}_0(x),\ldots,\widehat{Q}_{n-1}(x)\}$
of \emph{generalized associated polynomials} by

\begin{align}
\widehat{Q}_0(x) &= \widehat{\tau}_0, \notag \\
\widehat{Q}_k(x) &= \widehat{\tau}_k \cdot x \cdot \widehat{Q}_{k-1}(x)
        - \widehat{a}_{k-1, k}\widehat{Q}_{k-1}(x)- \widehat{a}_{k-2, k}\widehat{Q}_{k-2}(x) - \cdots \widehat{a}_{0, k}\widehat{Q}_{0}(x) 
\end{align}
with
\begin{align}
\widehat{\tau}_k &= \tau_{n-k}, \qquad k=0,\ldots,n-1   \notag \\
\widehat{a}_{j, k} &= \frac{\tau_{n-k}}{\tau_{n-j}} a_{n-k,n-j},
     \qquad k=1,\ldots,n-1, \quad j=0,\ldots,k-1.
    \label{eq:hata}
\end{align}
We will see that the generalized associated polynomials determine the structure of $R^{-1}$ where $R$ in our case is the polynomial Vandermonde-like matrices.
\subsection{Inversion formula}
By \cite{KO97}, a general inversion formula for Vandermonde-like
matrices satisfying the displacement equation (\ref{eq:dispwq}) is given by:
\begin{equation}
\label{eq:invformula}
R^{-1} = \tilde{I} \cdot \sum_{i=1}^{\alpha} 
    (\sum_{k=1}^{n} d_{ik} (W_{\widehat{Q}}^T)^{k-1}) \cdot
    V_{\widehat{Q}}^T \cdot \diag(c_i),
\end{equation}
where
\begin{equation}
\label{eq:cidik}
\left[\begin{matrix}c_1 & c_2 & \cdots & c_{\alpha}\end{matrix}\right] = 
D_x R^{-T} B^T, \quad  
\left[\begin{matrix} d_{ik} \end{matrix}\right] = G^TR^{-T}\tilde{I}(S_{P\widehat{Q}})^{-1}
\end{equation}
Here $S_{P\widehat{Q}}$ is the basis transformation matrix for passing from the basis of generalized associated polynomials $\widehat{Q}$ to
the power basis $P = \{1, x, \ldots, x^{n-1}\}$.
In the general case, this formula has complexity of $\mathcal{O}(\alpha n^3)$
which is not any better than the complexity of the Gaussian elimination
algorithm.
However, we will see that in the case when polynomials $Q$
satisfies the recurrence relations, this formula gives
the complexity of $\mathcal{O}(\alpha n^2)$ operations. To implement
this formula, we need the following three components:

\begin{enumerate}
\item To compute $c_i$ and $d_{ik}$, we need to solve $\alpha$
linear systems with matrix $R$. This problem will be addressed in sections
\ref{sec:gepp} and \ref{sec:gen}.

\item To compute $S_{P\widehat Q}$, we need to find recurrence relations
for generalized associated polynomials $\widehat Q$ and then use
them to calculate columns of $S_{P \widehat Q}$. This will be
covered in sections \ref{sec:assoc} and \ref{sec:basis}, respectively.
Multiplying by $S_{P \widehat Q}^{-1}$ is not a problem
since $S_{P\widehat Q}$ is a triangular matrix.

\item Finally, to compute all elements of $R^{-1}$ using (\ref{eq:invformula}),
we need a fast way to compute the expression
$(\sum_{k=1}^{n} d_{ik} (W_{\widehat{Q}}^T)^{k-1}) \cdot
V_{\widehat{Q}}^T $, which is what section \ref{sec:sums}
is dedicated to.
\end{enumerate}

\subsection{Quasiseparable matrices and polynomials}
In this section we briefly define quasiseparable matrices and their sub classes called semiseparable and well-free matrices and also the corresponding polynomial families.   
\bd 
\begin{itemize}
\item A matrix $A=[a_{i,j}]$ is called $(H,m)$-quasiseparable (i.e. Hessenberg lower part and order $m$ upper part) if {\bf (i)} it is strongly upper Hessenberg (i.e. non zero of first subdiagonal, $a_{i+1, i} \neq 0$ for $i=1,2, \cdots, n-1$ ), and {\bf (ii)} max (rank $A_{12}$) = m where the maximum is taken over all symmetric partitions of the form 
\[
A=\left[\begin{array}{c|c}
* & A_{12}\\
\hline
* & *
\end{array}\right]
\]
\item Let $A=[a_{i,j}]$ be a $(H,m)$-quasiseparable matrix. For $\lambda_i=\frac{1}{a_{i+1,i}}$, then the system of polynomials related to $A$ via 
\[
Q_k(x)=\lambda_1\lambda_2\cdots\lambda_k \:{\rm det} (xI-A)_{k \times k}
\]
is called a system of $(H,m)$-quasiseparable polynomials.
\end{itemize}
\ed
\begin{example}
{\bf (Tridiagonal matrices are $(H,1)$-quasiseparable)} 
It is known that real-orthogonal polynomials $\{Q_k(x)\}$ satisfy a three-term recurrence relation of the form
\be
Q_k(x)=(\alpha_kx-\delta_k)Q_{k-1}(x) - \gamma_k\cdot Q_{k-2}(x), \hspace{.3in} \alpha_k \neq 0, \gamma_k>0. 
\label{3toprr}
\ee
The real orthogonal polynomials satisfying (\ref{3toprr}) are related to the irreducible tridiagonal matrix
\be
T=\begin{bmatrix}
\frac{\delta_1 }{\alpha_1 } &\frac{\gamma_2 }{\alpha_2 }  & 0 & \cdots & 0\\ 
 \frac{1}{\alpha_1 }&  \frac{\delta_2 }{\alpha_2 }&\ddots  &\ddots  &\vdots  \\ 
0 & \frac{1}{\alpha_2 } &  \ddots& \frac{\gamma_{n-1} }{\alpha_{n-1} }  & 0\\ 
 \vdots &\ddots  & \ddots & \frac{\delta_{n-1} }{\alpha_{n-1} } &\frac{\gamma_n }{\alpha_n }  \\ 
 0& \cdots  & 0 &  \frac{1}{\alpha_{n-1} }& \frac{\delta_n }{\alpha_n }
\end{bmatrix}
\label{tdm}
\ee
 via
\[
Q_k(x)=\alpha_1\alpha_2\cdots\alpha_k \:{\rm det} (xI-A)_{k \times k}.
\]  
Note that if $A$ corresponds to the tridiagonal matrix $T$ (\ref{tdm}), then the corresponding submatrix $A_{12}$ has rank one as it is of the form $\left(\frac{\gamma_j}{\alpha_j}\right)e_k\:e_1^T$. Hence the tridiagonal matrices are $(H,1)$-qusiseparable. 
\end{example}
\begin{example}
{\bf (Unitary Hessenberg matrices are $(H,1)$-qusiseparable)} It is also known that Szeg${\rm \ddot{o}}$ polynomials $\{\phi^{\#}_k(x)\}$ satisfy a two-term recurrence relation of the form
\be
\left[ \begin{array}{c}
\phi_0(x)\\ 
\phi_0^{\#}(x)
\end{array}\right ] = \frac{1}{\mu_0}\left[ \begin{array}{c}
1\\ 
1
\end{array}\right ],
\hspace{.3in}
\left[ \begin{array}{c}
\phi_k(x)\\ 
\phi_k^{\#}(x)
\end{array}\right ] = \frac{1}{\mu_k}\left[ \begin{array}{cc}
1 & -\rho_k^*\\ 
 -\rho_k& 1 
\end{array}\right ]\left[ \begin{array}{c}
\phi_{k-1}(x)\\ 
x\:\phi_{k-1}^{\#}(x)
\end{array}\right ], \hspace{.3in} (k=1,2, \cdots,n) 
\label{2tsprr}
\ee
where $\{\rho_k\}$ are reflection coefficients satisfying $\rho_0=-1, \:\: |\rho_k| < 1({\rm for\:} k =1,2,\cdots,n-1), \:\: |\rho_n|\leq1$, $\{\mu_k\}$ are complementary parameters defined via $\mu_k=\left\{\begin{matrix}
\sqrt{1-|\rho_k|^2}, &|\rho_k|<1\\ 
1, &|\rho_k|=1
\end{matrix}\right.$ and $\{\phi_k(x)\}$ is a system of axillary polynomials.  The Szeg${\rm \ddot{o}}$ polynomials satisfying (\ref{2tsprr}) are related to the Unitary Hessenbery matrix
\be
H=\begin{bmatrix}
-\rho_0^*\rho_1 & -\rho_0^*\mu_1\rho_2 & \cdots & -\rho_0^*\mu_1\cdots\mu_{n-1}\rho_n\\ 
 \mu_1&  -\rho_1^*\rho_2& \cdots & -\rho_1^*\mu_2\cdots\mu_{n-1}\rho_n\\ 
 0& \ddots &  \ddots& \vdots\\ 
 \vdots& \ddots & \ddots & \ddots\\ 
 0 & 0 & \mu_{n-1} & -\rho_{n-1}^*\rho_n
\end{bmatrix}
\label{uHm}
\ee
via
\[
\phi_k^{\#}(x) =\frac{1}{\mu_1\mu_2\cdots\mu_k} \:{\rm det} (xI-H)_{k \times k}.
\]  
Note also that if $A$ corresponds to the Unitary Hessenberg matrix $H$ (\ref{uHm}), then the corresponding $3 \times (n-1)$ submatrix $A_{12}$ has the form
\[
A_{12}=\begin{bmatrix}
-\rho_k\mu_{k-1}\cdots\mu_3\mu_2\mu_1\rho_0^* & -\rho_{k-1}\mu_{k-2}\cdots\mu_3\mu_2\mu_1\rho_0^* & \cdots & -\rho_n\mu_{n-1}\cdots\mu_3\mu_2\mu_1\rho_0^*\\
-\rho_k\mu_{k-1}\cdots\mu_3\mu_2\rho_1^* & -\rho_{k-1}\mu_{k-2}\cdots\mu_3\mu_2\rho_1^* & \cdots & -\rho_n\mu_{n-1}\cdots\mu_3\mu_2\rho_1^*\\
-\rho_k\mu_{k-1}\cdots\mu_3\rho_2^* & -\rho_{k-1}\mu_{k-2}\cdots\mu_3\rho_2^* & \cdots & -\rho_n\mu_{n-1}\cdots\mu_3\rho_2^*
\end{bmatrix}
\]
which is also rank one. This is also true for all other symmetric partitions of $H$. Hence the Unitary Hessenberg matrices are $(H,1)$-qusiseparable. 
\end{example}
\subsubsection{Recurrence relations on quasiseparable, semi-separable and well-free polynomials}
It was first proved in \cite{EGO05} that the quasiseparable matrices corresponding to the system of quasiseparable polynomials $Q = \{Q_0(x), Q_1(x), \ldots, Q_{n-1}(x)\}$ with $\deg Q_k = k$  satisfies the {\it EGO-type two term recurrence relation}:
\begin{equation}
\left[ \begin{array}{c} G_k(x) \\ Q_k(x)  \end{array} \right]
= 
\left[\begin{array}{cc} \alpha_k & \beta_k \\ \gamma_k & \delta_k x + \theta_k \end{array} \right]
\left[\begin{array}{c} G_{k-1}(x) \\ Q_{k-1}(x) \end{array} \right].
\label{eq:ego}
\end{equation}
By referring to the classification paper \cite{BEGO13}, it is possible to generalize the recurrence relations of orthogonal polynomials  (\ref{3toprr}), which we called the generalized three-term recurrence relation or recurrence relations on {\it well-free polynomials} $Q$ defined via
\be
Q_k(x)=(\alpha_kx-\delta_k)Q_{k-1}(x) - (\beta_kx+\gamma_k)\cdot Q_{k-2}(x). 
\label{ewfrr}
\ee
where $Q = \{Q_0(x), Q_1(x), \ldots, Q_{n-1}(x)\}$ and $\deg Q_k = k$.
\newline\newline
It is also possible to generalize the recurrence relations of Szeg${\rm \ddot{o}}$ polynomials (\ref{2tsprr}) which we called Szeg${\rm \ddot{o}}$-type two-term recurrence relation or recurrence relations on {\it semiseparable polynomials} $Q$ defined via
\begin{equation}
\left[ \begin{array}{c} G_k(x) \\ Q_k(x)  \end{array} \right]
= 
\left[\begin{array}{cc} \alpha_k & \beta_k \\ \gamma_k & 1 \end{array} \right]
\left[\begin{array}{c} G_{k-1}(x) \\ (\delta_k x + \theta_k) Q_{k-1}(x) \end{array} \right],
\label{eq:semi}
\end{equation}
where $Q = \{Q_0(x), Q_1(x), \ldots, Q_{n-1}(x)\}$ with $\deg Q_k = k$ and $\{G_k(x)\}$ are axillary polynomials.
\subsection{Main results}
In Section \ref{sec:gepp}, we introduce a generalized fast Gaussian-elimination algorithm for the polynomial Vandermonde-like matrices. Next in section \ref{sec:gen} we introduce displacement operators for quasiseparable, semiseparable and well-free Vandermonde matrices. In section \ref{sec:assoc} we derive recurrence relations for Horner polynomials corresponding to the quasiseparable, semiseparable and well-free polynomials. In section \ref{sec:basis} we introduce the recurrence relations for the columns of basis transformation matrices for quasiseparable, semiseparable and well-free polynomials. Next in section \ref{sec:sums} we state supporting results which enable us to compute the sums in the inversion formula of the quasiseparable Vandermonde-like matrices. Finally in section \ref{sec:invalgo} we state the algorithm for the inversion of quasiseparable Vandermonde-like matrices and conclude that the algorithm costs only $\mathcal{O}(n^2)$ operations.

\section{Fast Gaussian elimination algorithm} 
\label{sec:gepp}
It is known that Gaussian elimination for matrices with
displacement structure can be implemented much faster than in $\mathcal{O}(n^3)$
operations. The following lemma from \cite{GKO95, KO97} serves as a basis for such an implementation.

\begin{lemma}\label{lemma:gaussdisp}
Let square matrix $R^{(k)}=\left[\begin{smallmatrix}d^{(k)} & u^{(k)}
    \\ l^{(k)} & R_{22}^{(k)}\end{smallmatrix}\right]$,
$d^{(k)} \neq 0$,
satisfy displacement equation
\[
\Omega^{(k)} R^{(k)} - R^{(k)} A^{(k)} = G^{(k)} B^{(k)}.
\]
Then the Schur complement $R^{(k+1)} = R_{22}^{(k)} - \frac{1}{d^{(k)}} l^{(k)} u^{(k)}$ satisfies
\[
    \Omega^{(k+1)} R^{(k+1)} - R^{(k+1)} A^{(k+1)} = G^{(k+1)} B^{(k+1)}.
\]
where $\Omega^{(k+1)}$ and $A^{(k+1)}$ are obtained from $\Omega^{(k)}$
and $A^{(k)}$, respectively, by removing the first row and column, and
\begin{eqnarray}
\left[\begin{matrix}0 \\ G^{(k+1)}\end{matrix} \right] = G^{(k)} - 
    \left[\begin{matrix}1 \\ \frac{1}{d^{(k)}}l^{(k)} \end{matrix}\right] \cdot g_1^{(k)}, \qquad
\left[\begin{matrix}0 & B^{(k+1)}\end{matrix}\right] = B^{(k)} - 
    b_1^{(k)} \cdot \left[\begin{matrix}1 & \frac{1}{d^{(k)}}u^{(k)} \end{matrix}\right],
\label{eq:schurgen}
\end{eqnarray}
where $g_1^{(k)}$ and $b_1^{(k)}$ are the first row of $G^{(k)}$
and the first column of $B^{(k)}$, respectively.
\end{lemma}

\noindent Lemma \ref{lemma:gaussdisp} guarantees that each Schur complement
obtained in the process of Gaussian elimination will have a similar
displacement structure. This allows us to compute only the generators
$G^{(k+1)}$, $B^{(k+1)}$ of the Schur complement instead of computing
all its entries. However, to obtain these generators using formulas
(\ref{eq:schurgen}), we need a way to compute the first row
$[\begin{smallmatrix}d^{(k)} & u^{(k)} \end{smallmatrix}]$
and the first column $\left[\begin{smallmatrix}d^{(k)} \\ l^{(k)}\end{smallmatrix}\right]$
of the matrix $R^{(k)}$.  In our case
\[
    \Omega^{(k)} = D_{\frac{1}{x}}^{(k)} = \diag\left(\frac{1}{x_k}, \frac{1}{x_{k+1}}, \ldots, \frac{1}{x_n}\right), \qquad
    A^{(k)} = W_Q^{(k)},
\]
where $W_Q^{(k)}$ is obtained from $W_Q$ by removing the first $(k-1)$ rows
and columns.
This means that Schur complements will satisfy
\begin{equation}\label{eq:schurdisp}
D_{\frac{1}{x}}^{(k)} R^{(k)} - R^{(k)} W_Q^{(k)} = G^{(k)} B^{(k)}.
\end{equation}
To get the first column of $R^{(k)}$, multiply (\ref{eq:schurdisp}) by $e_1$
from the right and by $D_x$ from the left to get
\begin{equation}\label{eq:geppcol}
\left[\begin{matrix}d^{(k)} \\ l^{(k)}\end{matrix}\right]
  = D_x \cdot G^{(k)} \cdot b_1^{(k)}.
\end{equation}
Obtaining the first row of $R^{(k)}$ is a bit harder. Multiply (\ref{eq:schurdisp})
from the left by $e_1^T \cdot D_x$ to get
\begin{equation}\label{eq:gepprowwq}
\left[\begin{matrix}d^{(k)} & u^{(k)} \end{matrix}\right] \cdot (I - x_k W_Q^{(k)})
    = x_k \cdot g_1^{(k)} \cdot B^{(k)}.
\end{equation}
Since $W_Q, M_Q$ and $N_Q$ are upper triangular matrices, we can write
$W_Q^{(k)} = N_Q^{(k)} \cdot ( M_Q^{(k)} )^{-1}$, where $N_Q^{(k)}$ and $M_Q^{(k)}$
are submatrices of $N_Q$ and $M_Q$, respectively, obtained by deleting the first $(k-1)$
rows and columns. Hence by multiplying (\ref{eq:gepprowwq}) by $M_Q^{(k)}$ from
the right one obtains
\begin{eqnarray}\label{eq:gepprow}
\left[\begin{matrix}d^{(k)} & u^{(k)} \end{matrix}\right] \cdot (M_Q^{(k)} - x_k N_Q^{(k)})
    = x_k \cdot g_1^{(k)} \cdot B^{(k)} \cdot M_Q^{(k)}.
\end{eqnarray}
To recover the first row
$[\begin{smallmatrix}d^{(k)} & u^{(k)} \end{smallmatrix}]$,
one has to perform a multiplication by $M_Q^{(k)}$
and solve a system of linear equations with
the matrix $(M_Q^{(k)} - x_k N_Q^{(k)})$. We will see in section \ref{sec:gen}
that in our case this matrix has a special quasiseparable structure
which allows us to do it in linear time.
\\
\begin{remark} {\bf (Partial pivoting)} 
It is known by \cite{GKO95} that row permutations of $R$ do not destroy the displacement
structure. This means that we can easily incorporate partial pivoting into
our Gaussian elimination algorithm. If $R$ satisfies the displacement equation
(\ref{eq:dispwq}), swapping its $i$-th and $j$-th rows is equivalent to
swapping $x_i$ and $x_j$, and the $i$-th and $j$-th rows of $G$.
\end{remark}
\vspace{.1in}
\noindent Now we can summarize the Gaussian elimination with partial pivoting
(GEPP) for Vandermonde-like matrices as Algorithm \ref{algo:gepp}.
This algorithm produces the factorization
\begin{equation}\label{eq:plu}
R = P \cdot L \cdot U, \qquad P = P_1 \cdot P_2 \cdots P_n,
\end{equation}
where $P_k$ is the permutation matrix corresponding to swapping
rows on the $k$-th step of GEPP.

\begin{algorithm}
\caption{GEPP for Vandermonde-like matrices}
\label{algo:gepp}
\begin{algorithmic}[1]
\State Let $G^{(1)} = G$, $B^{(1)} = B$.
\For{ $k = 1, 2, \ldots, n$ }
\State Compute the first column
        $\left[\begin{smallmatrix}d^{(k)} \\ l^{(k)}\end{smallmatrix}\right]$
        using (\ref{eq:geppcol}).
\State Find, say at position $m$, the maximum magnitude element of
$\left[\begin{smallmatrix}d^{(k)} \\ l^{(k)}\end{smallmatrix}\right]$.
\State Swap $d^{(k)}$ and $m$-th entry of
    $\left[\begin{smallmatrix}d^{(k)} \\ l^{(k)}\end{smallmatrix}\right]$.
\State Swap $x_k$ and $x_{k+m-1}$.
\State Swap the first and $m$-th rows of $G^{(k)}$.
\State Compute the right-hand side of (\ref{eq:gepprow}).
\State Solve (\ref{eq:gepprow}) for
$[\begin{smallmatrix}d^{(k)} & u^{(k)} \end{smallmatrix}]$.
\State Write the $k$-th column
$\left[\begin{smallmatrix}0 \\ \vdots \\ 0 \\ 1 \\ l^{(k)}/d^{(k)} \end{smallmatrix}\right]$ of $L$.
\State Write the $k$-th row 
$[\begin{smallmatrix} 0 & \cdots & 0 & d^{(k)} & u^{(k)} \end{smallmatrix}]$ of $U$.
\State Let $P_k$ be a $n \times n$ matrix which swaps
$k$-th and $(k+m-1)$-th rows.
\State Compute $G^{(k+1)}$ and $B^{(k+1)}$ using (\ref{eq:schurgen}).
\EndFor
\end{algorithmic}
\end{algorithm}

\noindent The following theorem follows from Algorithm \ref{algo:gepp}.

\bt
\label{thm:gepp}
Let $R$ satisfy the displacement equation (\ref{eq:dispwq}).
Let $C_1(n)$ be an upper bound on the complexity of solving
a linear system with a matrix of the form $(M_Q^{(k)} - x_k N_Q^{(k)})$,
  $k = 1,\ldots,n$,
and $C_2(n)$ an upper bound on complexity of performing a vector-matrix
multiplication $v \cdot M_Q^{(k)}$, $k=1,\ldots,n$. Then the factorization
(\ref{eq:plu}) can be computed in $\mathcal{O}(\alpha n^2) + nC_1(n) + nC_2(n)$
operations.
\et
\begin{proof}
Consider the loop of Algorithm \ref{algo:gepp}.
Steps 3 and 13 can be done in $\mathcal{O}(\alpha n)$ operations.
Steps 4--7 and 10--11 can be performed in $\mathcal{O}(n)$ operations.
Step 8 can be done in $\mathcal{O}(\alpha n)$ + $C_2(n)$ operations,
    and step 9 in $C_1(n)$ operations. Since $n$ iterations are performed,
    the overall complexity is $\mathcal{O}(\alpha n^2) + nC_1(n) + nC_2(n)$ operations.
\end{proof}
\noindent We will see in section \ref{sec:gen} that in our case, matrices of the form
$M_Q - \xi N_Q$ (where $\xi$ is a constant) are quasiseparable.
For such matrices, linear-time inversion and multiplication algorithms are available
(see \cite{EG99}), thus $C_1(n) = \mathcal{O}(n)$ and $C_2(n) = \mathcal{O}(n)$.


\section{Recurrence relation matrices} \label{sec:gen}
In this section, we study the structure of recurrence relation
matrices $M_Q$ and $N_Q$ for quasiseparable, semiseparable
and well-free polynomials. These matrices correspond to the displacement operator for polynomial Vandermonde-like matrices. In a later section, we use the structures of $M_Q$ and $N_Q$ to compute generators
for matrices of the form $M_Q - \xi N_Q$. 
We need the latter for an efficient Gaussian elimination
algorithm for Vandermonde-like matrices.

\subsection{Quasiseparable polynomials}
The next two lemmas show that for a system of polynomials $\{Q\}$ of $1$--quasiseparable polynomials,
recurrence relation matrix $M_Q$ is upper triangular $2$--quasiseparable. As $M_Q$ is an upper triangular matrix with quasiseparable structure instead of calling it $(H,2)$-quasiseparable, we called $M_Q$ a $2$--quasiseparable matrix. Similarly we call $M_Q$ a $2$--quasiseparable matrix for the semiseprable and well-free cases.  
\begin{lemma}
\label{lemma:mnqs}
If a system of quasiseparable polynomials $\{Q\}$ satisfies recurrence
relations (\ref{eq:ego}), then the recurrence relation matrices (\ref{eq:mn})
have the form
\begin{equation}
M_Q=\left[
    \begin{array}{cccccc}
           1 & -\theta_1 & a_{0,2} & a_{0,3} & \cdots &   a_{0,n-1} \\
           0 &       1 & -\theta_2 &   a_{1,3} & \ldots &   a_{1,n-1} \\
      \vdots &  \ddots &  \ddots &  \ddots    &  \ddots&      \vdots \\
 \vdots &   &  \ddots &  \ddots    & \ddots &      a_{n-3,n-1} \\

      \vdots &         &   & \ddots  &  1 & -\theta_{n-1} \\
           0 &  \cdots &  \cdots & \cdots &    0 &           1
    \end{array}
\right],
\quad
N_Q=\left[
    \begin{array}{ccccc}
         0 & \delta_1 &        0 & \cdots &            0 \\
         0 &        0 & \delta_2 & \ddots &            \vdots \\
         0 &        \ddots &        0 & \ddots &       0 \\
    \vdots &          &       \ddots   & \ddots & \delta_{n-1} \\
         0 &   \cdots &   \cdots &      0 &            0
    \end{array}
\right],
\label{eq:mnqs}
\end{equation}
where
\begin{equation}
a_{jk} = -\beta_{j+1} \alpha_{j+1,k}^{\times} \gamma_k,
    \qquad \alpha_{j+1,k}^{\times}=\prod_{i=j+2}^{k-1}\alpha_{i}.
\label{eq:qsaij}
\end{equation}
\end{lemma}
\begin{proof}
From \cite{BOZ11} (see the proof of Theorem 3.5),
we know that polynomials $Q_k(x)$ have the form
\begin{equation}
\begin{split}
Q_k(x) = (\delta_k x + \theta_k) Q_{k-1}(x) + \gamma_k \beta_{k-1} Q_{k-2}(x) 
    + \gamma_k \alpha_{k-1} \beta_{k-2} Q_{k-3}(x) \\
    + \gamma_k \alpha_{k-1} \alpha_{k-2} \beta_{k-3} Q_{k-4}(x) + \ldots
    + \gamma_k \alpha_{k-1}\cdots \alpha_2 \beta_1 Q_0(x).
\end{split}
\end{equation}
Comparing this with the recurrence relations (\ref{generalrecur}),
one can see that matrices $M_Q$ and $N_Q$ are defined by (\ref{eq:mnqs}).
\end{proof}

\begin{lemma} \label{lemma:mgen}
For any $t_1, t_2, \ldots, t_{n-1}$, the matrix
\begin{equation}
\tilde{M}_Q(t_1,t_2,\ldots,t_{n-1}) = \left[
     \begin{array}{cccccc}
           1 & t_1 & a_{0,2} & a_{0,3} & \cdots &   a_{0,n-1} \\
           0 &       1 & t_2 &   a_{1,3} & \ldots &   a_{1,n-1} \\
      \vdots &  \ddots &  \ddots &  \ddots    &  \ddots&      \vdots \\
 \vdots &   &  \ddots &  \ddots    & \ddots &      a_{n-3,n-1} \\

      \vdots &         &   & \ddots  &  1 & t_{n-1} \\
           0 &  \cdots &  \cdots & \cdots &    0 &           1
    \end{array}
\right]
\end{equation}
with
\[
a_{jk} = -\beta_{j+1} \alpha_{j+1,k}^{\times} \gamma_k,
    \qquad \alpha_{j+1,k}^{\times}=\alpha_{j+2}\alpha_{j+3}\cdots \alpha_{k-1}.
\]
is upper triangular $2$--quasiseparable with generators
\begin{eqnarray*}
d_j &=& 1, \quad j=1,\ldots,n, \\
g_j &=& [\,t_j \;\;\; {-\beta_j}\,], \quad j = 1,\ldots,n-1, \\
b_i &=& \left[\begin{array}{cc}
            0           & 0         \\
            \gamma_i    & \alpha_i
        \end{array}\right], \quad i = 2, \ldots, n-1,\\
h_k &=& [\,1 \;\; 0\,]^T, \quad k = 2, \ldots, n.
\end{eqnarray*}
That is,
\begin{equation}
[\tilde{M}_Q(t_1,t_2,\ldots,t_{n-1})]_{jk} = \begin{cases}
    0, & k < j, \\
    d_j, & k = j, \\
    g_j h_k, & k = j + 1, \\
    g_j b_{j+1} b_{j+2} \cdots b_{k-1} h_k, & k > j + 1.
\end{cases} \label{eq:mqtij}
\end{equation}
\end{lemma}
\begin{proof}
It is easy to check that the formula (\ref{eq:mqtij}) holds for $k\leq j+1$.
If $k > j + 1$, then
\[
    b_{j+1}b_{j+2}\cdots b_{k-1} = \left[\begin{array}{cc}
        0 & 0 \\
        \alpha_{j,k-1}^{\times} \gamma_{k-1} & \alpha_{j,k}^{\times} 
    \end{array}\right].
\]
Therefore, $g_j b_{j+1}b_{j+2}\cdots b_{k-1} h_k
    = -\beta_j \alpha_{j,k-1}^{\times} \gamma_{k-1} $.
On the other hand, 
\[
[\tilde{M}_Q(t_1,t_2,\ldots,t_{n-1})]_{jk} = a_{j-1,k-1}
 = -\beta_j \alpha_{j,k-1}^{\times}\gamma_{k-1}.
\]
Hence the formula (\ref{eq:mqtij}) holds for $k > j + 1$ as well.
\end{proof}
\noindent Using Lemma \ref{lemma:mgen} with $t_i = -\theta_i$,
we conclude that $M_Q$ is $2$--quasiseparable.
By letting $t_i = -\theta_i - \xi \delta_i$,
we can obtain generators for the matrix $(M_Q - \xi N_Q)$.

\subsection{Semiseparable polynomials}
In this section we obtain the structure of the recurrence relation matrices
$M_Q$ and $N_Q$ for semiseparable polynomials defined by (\ref{eq:semi}).
For convenience we denote $\beta_0 = 1$.

\begin{lemma}
\label{lemma:mnss}
If a system of semiseparable polynomials $\{Q\}$ satisfies recurrence
relations (\ref{eq:semi}), then the recurrence relation matrices (\ref{eq:mn})
have the form
\begin{equation}
\begin{split}
M_Q&=\left[
    \begin{array}{cccccc}
           1 & -(\theta_1 + \gamma_1 \beta_0) & a_{0,2}  & a_{0,3}& \cdots &   a_{0,n-1} \\
           0 &       1 & -(\theta_2 + \gamma_2 \beta_1) & a_{1,3} &\ldots &   a_{1,n-1} \\
      \vdots &  \ddots &  \ddots & \ddots    &  \ddots &      \vdots \\
 \vdots &   &  \ddots &  \ddots   & \ddots  &      a_{n-3,n-1} \\
      \vdots &         &   & \ddots & 1 & -(\theta_{n-1} + \gamma_{n-1} \beta_{n-2}) \\
           0 &  \cdots &  \cdots &  \cdots  &  0 &           1
    \end{array}
\right],\\
N_Q&=\left[
    \begin{array}{ccccc}
         0 & \delta_1 &        0 & \cdots &            0 \\
         0 &        0 & \delta_2 & \ddots &            \vdots \\
         0 &         &        \ddots & \ddots &      0  \\
    \vdots &          &          & \ddots & \delta_{n-1} \\
         0 &   \cdots &   \cdots &      0 &            0
    \end{array}
\right],
\end{split}
\label{eq:mnss}
\end{equation}
where
\begin{equation}
a_{jk} = -\beta_{j} \cdot (\alpha - \beta \gamma)_{j,k}^{\times} \cdot \gamma_k,
\label{eq:aijss}
\end{equation}
\[
    \qquad (\alpha - \beta \gamma)_{j,k}^{\times}= \prod_{i=j+1}^{k-1}
    (\alpha_i - \beta_i \gamma_i).
\]
\end{lemma}
\begin{proof}
From \cite{BOZ11} (see the proof of Theorem 4.7) we know that semiseparable polynomials defined by recurrence relations
(\ref{eq:semi}) are quasiseparable and satisfy the following recurrence relations
of the form (\ref{eq:ego}):
\begin{equation*}
\left[ \begin{array}{c} \tilde{G}_k(x) \\ Q_k(x)  \end{array} \right]
= 
\left[\begin{array}{cc}
    \alpha_{k-1} - \beta_{k-1}\gamma_{k-1} &
    \beta_{k-1} \\
    \gamma_k(\alpha_{k-1} - \beta_{k-1}\gamma_{k-1}) &
    \delta_k x + \theta_k + \gamma_k \beta_{k-1}
    \end{array} \right]
\left[\begin{array}{c} \tilde{G}_{k-1}(x) \\ Q_{k-1}(x) \end{array} \right]
\end{equation*}
with $\tilde{G}_k(x) = G_{k-1}(x),\;\; \tilde{G}_0(x) = 0$.
Now we can apply Lemma \ref{lemma:mnqs} to obtain (\ref{eq:mnss}).
\end{proof}
\noindent The next result is a direct consequence of Lemma \ref{lemma:mgen} and Lemma \ref{lemma:mnss}.

\begin{corollary} \label{lemma:mgenss}
For any $t_1, t_2, \ldots, t_{n-1}$, the matrix
\begin{equation}
\tilde{M}_Q(t_1,t_2,\ldots,t_{n-1}) = \left[
   \begin{array}{cccccc}
           1 & t_1 & a_{0,2} & a_{0,3} & \cdots &   a_{0,n-1} \\
           0 &       1 & t_2 &   a_{1,3} & \ldots &   a_{1,n-1} \\
      \vdots &  \ddots &  \ddots &  \ddots    &  \ddots&      \vdots \\
 \vdots &   &  \ddots &  \ddots    & \ddots &      a_{n-3,n-1} \\

      \vdots &         &   & \ddots  &  1 & t_{n-1} \\
           0 &  \cdots &  \cdots & \cdots &    0 &           1
    \end{array}
\right]
\end{equation}
with
\[
a_{jk} = -\beta_{j} \cdot (\alpha - \beta \gamma)_{j,k}^{\times} \cdot
        \gamma_k,
\]
\[
    \qquad (\alpha - \beta \gamma)_{j,k}^{\times}= \prod_{i=j+1}^{k-1}
    (\alpha_i - \beta_i \gamma_i).
\]
is upper triangular $2$--quasiseparable with generators
\begin{eqnarray*}
d_j &=& 1, \quad j=1,\ldots,n, \\
g_j &=& [\,t_j \;\;\; {-\beta_{j-1}}\,], \quad j = 1,\ldots,n-1, \\
b_i &=& \left[\begin{array}{cc}
            0           & 0         \\
            \gamma_i(\alpha_{i-1} - \beta_{i-1}\gamma_{i-1}) & \alpha_{i-1} - \beta_{i-1}\gamma_{i-1}
        \end{array}\right], \quad i = 2, \ldots, n-1\\
h_k &=& [\,1 \;\; 0\,]^T, \quad k = 2, \ldots, n.
\end{eqnarray*}
\end{corollary}
\noindent Using Corollary \ref{lemma:mgenss} with $t_i = -(\theta_i+\gamma_{i}\beta_{i-1})$,
we conclude that semiseparable recurrent matrix $M_Q$ is $2$--quasiseparable.
By letting $t_i = -(\theta_i+\gamma_{i}\beta_{i-1}) - \xi \delta_i$,
we can obtain generators for the matrix $(M_Q - \xi N_Q)$ of semiseparable polynomials $Q$.

\subsection{Well-free polynomials}
In this section we obtain the structure of the recurrence relation matrices
$M_Q$ and $N_Q$ for well-free polynomials defined by (\ref{ewfrr}).
For convenience we denote $\alpha_0=1$ and $\beta_1=0$.
\begin{lemma}
\label{lemma:emnwf}
If a system of well-free polynomials $\{Q\}$ satisfies recurrence
relations (\ref{ewfrr}), then the recurrence relation matrices (\ref{eq:mn})
have the form
\begin{equation}
M_Q=\left[
    \begin{array}{cccccc}
           1 & \delta_1 + \frac{\beta_1}{ \alpha_0} & a_{0,2}  & a_{0,3}& \cdots &   a_{0,n-1} \\
           0 &       1 & \delta_2 + \frac{\beta_2}{ \alpha_1} & a_{1,3} &\ldots &   a_{1,n-1} \\
      \vdots &  \ddots &  \ddots & \ddots    &  \ddots &      \vdots \\
 \vdots &   &  \ddots &  \ddots   & \ddots  &      a_{n-3,n-1} \\
      \vdots &         &   & \ddots & 1 & \delta_{n-1} + \frac{\beta_{n-1}}{ \alpha_{n-2}} \\
           0 &  \cdots &  \cdots &  \cdots  &  0 &           1
    \end{array}
\right],\quad
N_Q=\left[
    \begin{array}{ccccc}
         0 & \alpha_1 &        0 & \cdots &            0 \\
         0 &        0 & \alpha_2 & \ddots &            \vdots \\
         0 &         &        \ddots & \ddots &      0  \\
    \vdots &          &          & \ddots & \alpha_{n-1} \\
         0 &   \cdots &   \cdots &      0 &            0
    \end{array}
\right],
\label{eq:mnwf}
\end{equation}
where
\begin{equation}
a_{jk} = \frac{\alpha_k}{\alpha_{j+2}}\left(\left( \frac{\delta_{j+1}}{\alpha_{j+1}}+\frac{\beta_{j+1}}{\alpha_j \alpha_{j+1}} \right)\beta_{j+2} + \gamma_{j+2} \right)\cdot \left(\frac{\beta}{\alpha}\right)_{j+1,k}^{\times},
\label{eq:aijwf}
\end{equation}
\[
    \qquad  \left(\frac{\beta}{\alpha}\right)_{j+1,k}^{\times}= \prod_{i=j+2}^{k-1}\frac{\beta_{i+1}}{\alpha_{i+1}}.
\]
\end{lemma}
\begin{proof}
From \cite{BEGO13} (see the proof of Theorem 4.4) we know that the well-free polynomials $Q_k(x)$ defined by recurrence relations (\ref{ewfrr}) are also quasiseparable satisfying the recurrence relation
\be
\begin{aligned}
Q_k(x)=\alpha_k\:x\:Q_{k-1}(x) & - \left(\delta_k + \frac{\beta_k}{\alpha_{k-1}} \right)Q_{k-1}(x)-(d_{k-1}\beta_k+\gamma_k)Q_{k-2}(x)\\
& -  g_{k-2}\beta_k Q_{k-3}(x)-g_{k-3}b_{k-2}\beta_kQ_{k-4}(x)-g_{k-4}b_{k-3}b_{k-2}\beta_kQ_{k-5}(x) \\
& - \cdots -  g_2b_3b_4\cdots b_{k-2} \beta_k Q_1(x)- g_1b_2b_3\cdots b_{k-2} \beta_k Q_0(x)
\end{aligned}
\label{eq:lwfqs}
\ee
where $d_k=\frac{\delta_k}{\alpha_k}+\frac{\beta_k}{\alpha_{k-1}\alpha_k}$ for $k=2,3,\cdots,n$, $g_k=\frac{d_k\beta_{k+1}+\gamma_{k+1}}{\alpha_{k+1}}$ for $k=1,2,\cdots,n-1$ and $b_k=\frac{\beta_{k+1}}{\alpha_{k+1}}$ for $k=2,3,\cdots,n-1$.
\\
Comparing the recurrence relation (\ref{eq:lwfqs}) with the recurrence relations (\ref{generalrecur}),
one can see that matrices $M_Q$ and $N_Q$ are defined by (\ref{eq:mnwf}).
\end{proof}
\noindent The next result is also a direct consequence of Lemma \ref{lemma:mgen} and Lemma \ref{lemma:emnwf}.

\begin{corollary} \label{lemma:mgenwf}
For any $t_1, t_2, \ldots, t_{n-1}$, the matrix
\begin{equation}
\tilde{M}_Q(t_1,t_2,\ldots,t_{n-1}) = \left[
   \begin{array}{cccccc}
           1 & t_1 & a_{0,2} & a_{0,3} & \cdots &   a_{0,n-1} \\
           0 &       1 & t_2 &   a_{1,3} & \ldots &   a_{1,n-1} \\
      \vdots &  \ddots &  \ddots &  \ddots    &  \ddots&      \vdots \\
 \vdots &   &  \ddots &  \ddots    & \ddots &      a_{n-3,n-1} \\

      \vdots &         &   & \ddots  &  1 & t_{n-1} \\
           0 &  \cdots &  \cdots & \cdots &    0 &           1
    \end{array}
\right]
\end{equation}
with
\[
a_{jk} = \frac{\alpha_k}{\alpha_{j+2}}\left(\left( \frac{\delta_{j+1}}{\alpha_{j+1}}+\frac{\beta_{j+1}}{\alpha_j \alpha_{j+1}} \right)\beta_{j+2} + \gamma_{j+2} \right)\cdot \left(\frac{\beta}{\alpha}\right)_{j+1,k}^{\times},
\]
\[
    \qquad  \left(\frac{\beta}{\alpha}\right)_{j+1,k}^{\times}= \prod_{i=j+2}^{k-1}\frac{\beta_{i+1}}{\alpha_{i+1}}.
\]
is upper triangular $2$--quasiseparable with generators
\begin{eqnarray*}
d_j &=& 1, \quad j=1,\ldots,n, \\
g_j &=& \left[\,t_j \;\;\; \frac{t_j{\beta_{j+1}}}{\alpha_j} + \gamma_{j+1}\, \right], \quad j = 1,\ldots,n-1, \\
b_i &=& \left[\begin{array}{cc}
            0           & 0         \\
            1    & \frac{\beta_{i+1}}{\alpha_i}
        \end{array}\right], \quad i = 2, \ldots, n-1,\\
h_k &=& [\,1 \;\; 0\,]^T, \quad k = 2, \ldots, n.
\end{eqnarray*}
\end{corollary}
\noindent Using Corollary \ref{lemma:mgenwf} with $t_i = \left(\delta_i+\frac{\beta_{i}}{\alpha_{i-1}} \right)$,
we conclude that well-free recurrent matrix $M_Q$ is $2$--quasiseparable.
By letting $t_i = \left(\delta_i+\frac{\beta_{i}}{\alpha_{i-1}} \right) - \xi \alpha_i$,
we can obtain generators for the matrix $(M_Q - \xi N_Q)$ of well-free polynomials $Q$.


\section{Recurrence relations for associated polynomials} \label{sec:assoc}
Fast Traub-like algorithms (see e.g. \cite{BEGOTZ10, BEGOT12, O01, GO94, CR93, P64, T66}) were derived by using the properties of so called associate or Horner polynomials. It is known that these Horner polynomials determine the structure of the inversion of polynomial Vandermonde matrices leading an efficient inversion algorithm with cost $\mathcal{O}(n^2)$ operations. 
\\\\
In this section we derive the recurrence relations for associated polynomials for the systems of quasiseparable, semiseparable and well-free polynomials. Later, we use these recurrence relations to compute the displacement operator $W_{\widehat{Q}}= N_{\widehat{Q}} \:M_{\widehat{Q}}^{-1}$, basis transformation matrix $S_{P\:\widehat{Q}}$, and Vandermonde matrix $V_{\widehat{Q}}$ for the the system of Horner polynomials $\{\widehat{Q}\}$. 
\begin{lemma}\label{lemma:hatqs}
Let $\{Q\}$ be a system of quasiseparable polynomials satisfying recurrence
relations (\ref{eq:ego}). Then the system of generalized associated polynomials
$\widehat{Q}=\{\widehat{Q}_0(x), \widehat{Q}_1(x),\cdots,\widehat{Q}_{n-1}(x)\}$ is also quasiseparable and satisfies recurrence relations
\begin{equation}
\left[ \begin{array}{c} \widehat{G}_k(x) \\ \widehat{Q}_k(x)  \end{array} \right]
= 
\left[\begin{array}{cc} \widehat{\alpha}_k & \widehat{\beta}_k \\
        \widehat{\gamma}_k & \widehat{\delta}_k x + \widehat{\theta}_k \end{array} \right]
\left[\begin{array}{c} \widehat{G}_{k-1}(x) \\ \widehat{Q}_{k-1}(x) \end{array} \right],
    \quad k = 1,\ldots,n-1
\label{eq:hatego}
\end{equation}
with
\begin{eqnarray}
\widehat{\alpha}_k &=& \alpha_{n-k+1}, \notag \\
\widehat{\beta}_k  &=& \frac{\gamma_{n-k+1}}{\delta_{n-k+1}}, \notag \\
\widehat{\gamma}_k &=& \beta_{n-k+1} \cdot \delta_{n-k}, \\
\widehat{\delta}_k &=& \delta_{n-k}, \notag \\
\widehat{\theta}_k &=& \frac{\delta_{n-k}}{\delta_{n-k+1}} \theta_{n-k+1}. \notag
\end{eqnarray}
\end{lemma}

\begin{proof}
From Lemma \ref{lemma:mnqs} we know that the quasiseparable polynomials $Q$ satisfy general recurrence
relations (\ref{generalrecur}) with $a_{jk}$ defined by (\ref{eq:qsaij})
for $j < k-1$ and $a_{k-1,k} = -\theta_k$.
By definition (\ref{eq:hata}), for $j < k-1$ we can write
\begin{equation*}
\begin{split}
\widehat{a}_{jk} = \frac{\delta_{n-k}}{\delta_{n-j}} a_{n-k,n-j}
=
-\frac{\gamma_{n-j}}{\delta_{n-j}}  \alpha_{n-k+1,n-j}^{\times} \beta_{n-k+1}\delta_{n-k}
= -\widehat{\beta}_{j+1} \widehat{\alpha}_{j+1,k}^{\times} \widehat{\gamma}_{k}.
\end{split}
\end{equation*}
For $j = k-1$,
\[
    \widehat{a}_{k-1,k} = \frac{\delta_{n-k}}{\delta_{n-k+1}} a_{n-k,n-k+1}
    = \frac{\delta_{n-k}}{\delta_{n-k+1}} \theta_{n-k+1} = -\widehat{\theta}_k
\]
Using Lemma \ref{lemma:mnqs} we conclude that the associated polynomials $\widehat{Q}$ satisfy (\ref{eq:hatego}).
\end{proof}
\noindent The next result shows the recurrence relations for associate polynomials corresponding to the semiseparable polynomials defined by recurrence relation (\ref{eq:semi}). 
\begin{lemma}\label{lemma:hatss}
Let $\{Q\}$ be a system of semiseparable polynomials satisfying recurrence
relations (\ref{eq:semi}). Then the system of generalized associated polynomials $\widehat{Q}=\{\widehat{Q}_0(x), \widehat{Q}_1(x),\cdots,\widehat{Q}_{n-1}(x)\}$ is also semiseparable and satisfies recurrence relations

\begin{equation}\label{eq:hatsemi}
\left[ \begin{array}{c} \widehat{G}_k(x) \\ \widehat{Q}_k(x)  \end{array} \right]
= 
\left[\begin{array}{cc} \widehat{\alpha}_k & \widehat{\beta}_k \\
        \widehat{\gamma}_k & 1 \end{array} \right]
\left[\begin{array}{c} \widehat{G}_{k-1}(x) \\
        \left(\widehat{\delta}_k x + \widehat{\theta}_k \right) \widehat{Q}_{k-1}(x) \end{array} \right],
\quad k = 1,\ldots,n-1
\end{equation}
with
\begin{eqnarray*}
\widehat{\alpha}_k &=& \alpha_{n-k}, \\
\widehat{\beta}_k &=& \frac{\gamma_{n-k}}{\delta_{n-k}}, \\
\widehat{\gamma}_k &=& \beta_{n-k}\cdot\delta_{n-k}, \\
\widehat{\delta}_k &=& \delta_{n-k}, \\
\widehat{\theta}_k &=& \frac{\delta_{n-k}}{\delta_{n-k+1}}\theta_{n-k+1}.
\end{eqnarray*}
\end{lemma}

\begin{proof}
From Lemma \ref{lemma:mnss} we know that the semiseparable polynomials $Q$ satisfy general recurrence
relations (\ref{generalrecur}) with $a_{jk}$ defined by (\ref{eq:aijss})
for $j < k - 1$ and $a_{k-1,k} = -(\theta_k + \gamma_k \beta_{k-1})$.
By definition (\ref{eq:hata}), for $j < k - 1$ we can write
\[
\widehat{a}_{jk} = \frac{\delta_{n-k}}{\delta_{n-j}} a_{n-k,n-j}
= -\frac{\gamma_{n-j}}{\delta_{n-j}}
(\alpha - \beta \gamma)^\times_{n-k,n-j}
\delta_{n-k}\beta_{n-k} 
= -\widehat{\beta_j} (\widehat{\alpha} - \widehat{\beta}\widehat{\gamma})^\times_{j,k}\widehat{\gamma_k}.
\]
For $j = k - 1$, we get
\[
\widehat{a}_{k-1,k} = \frac{\delta_{n-k}}{\delta_{n-k+1}} a_{n-k,n-k+1}
    = -\frac{\delta_{n-k}}{\delta_{n-k+1}} (\theta_{n-k+1} + \gamma_{n-k+1}\beta_{n-k}) 
    = -(\widehat\theta_k + \widehat\gamma_k\widehat\beta_{k-1} ).
\]
From Lemma \ref{lemma:mnss} we conclude that the associated polynomials $\widehat{Q}$ satisfy (\ref{eq:hatsemi}).
\end{proof}
\noindent The following shows the recurrence relations for associate polynomials corresponding to the well-free polynomials defined by recurrence relation (\ref{ewfrr}).

%

\begin{lemma}\label{lemma:hatwf}
Let $\{Q\}$ be a system of quasiseparable polynomials satisfying recurrence
relations (\ref{eq:ego}) where $\beta_k \neq 0$.  
Then the system of generalized associated polynomials
$\widehat{Q}$ satisfies recurrence relations
\begin{eqnarray}
\widehat{Q}_1(x) &=& (\widehat{\alpha}_1 x - \widehat{\delta}_1)\widehat{Q}_0(x), \nonumber \\
\widehat{Q}_k(x) &=& (\widehat{\alpha}_k x - \widehat{\delta}_k) \widehat{Q}_{k-1}(x)
              - (\widehat{\beta}_k x + \widehat{\gamma}_k)\widehat{Q}_{k-2}(x), 
\quad k\geq 2
\label{hatwfrr}
\end{eqnarray}
with
\begin{eqnarray}\label{eq:wfcoef}
\widehat{\alpha}_k &=& \delta_{n-k}, \notag \\
\widehat{\beta}_k &=& \delta_{n-k}\alpha_{n-k+2}\frac{\beta_{n-k+1}}{\beta_{n-k+2}}, \notag \\
\widehat{\delta}_k &=& -\frac{\delta_{n-k}}{\delta_{n-k+1}}\left(\theta_{n-k+1} +
                \alpha_{n-k+2}\frac{\beta_{n-k+1}}{\beta_{n-k+2}}\right), \\
\widehat{\gamma}_k &=& \frac{\delta_{n-k}}{\delta_{n-k+2} } \cdot
                        \frac{\beta_{n-k+1}}{\beta_{n-k+2}} \cdot
                        ( \theta_{n-k+2}\alpha_{n-k+2} - 
                             \beta_{n-k+2} \gamma_{n-k+2} ). \notag
\end{eqnarray}
\end{lemma}
\begin{proof}
Since $\{Q\}$ is the system of quasiseparable polynomials satisfying (\ref{eq:ego}) then by Lemma \ref{lemma:hatqs} and Lemma \ref{lemma:mnqs} we know that its associated system of polynomials $\{\widehat{Q}\}$ satisfy 
\begin{equation}\label{eq:genrechat}
\widehat Q_k(x) = \widehat\alpha_k \cdot x \cdot \widehat Q_{k-1}(x)
                  - \widehat a_{k-1,k}\widehat Q_{k-1}(x)- \widehat a_{k-2,k}\widehat Q_{k-2}(x)-\cdots-\widehat a_{0,k}\widehat Q_{0}(x)
\end{equation}
with 
\begin{equation}\label{eq:ajkwant}
\widehat{a}_{jk} = -\frac{\gamma_{n-j}}{\delta_{n-j}} \cdot
    \left(\prod_{i=j+2}^{k-1} \alpha_{n-i+1} \right) \cdot
    \beta_{n-k+1}\delta_{n-k}, \quad j < k -1,
\end{equation}
and
\begin{equation}\label{eq:ajkwant2}
\widehat{a}_{k-1,k} = -\frac{\delta_{n-k}}{\delta_{n-k+1}}\theta_{n-k+1}.
\end{equation}
By Lemma \ref{lemma:emnwf}, if a system of polynomials $\{\widehat{Q}\}$ satisfies 3-term recurrence relations of the form (\ref{ewfrr}) having generators $\widehat\alpha_k, \widehat\beta_k, \widehat\gamma_k$ and $\widehat\delta_k$  then the system satisfies the general recurrence relation of the form (\ref{eq:genrechat}) with
\begin{equation}
    \widehat a_{jk} = \frac{\widehat\alpha_k}{\widehat\alpha_{j+2}}
    \left[
        \left( \frac{\widehat\delta_{j+1}}{\widehat\alpha_{j+1}}
            + \frac{\widehat\beta_{j+1}}{\widehat\alpha_j\widehat\alpha_{j+1}}
        \right)\widehat\beta_{j+2} + \widehat\gamma_{j+2}
    \right]
    \cdot
    \prod_{i = j+2}^{k-1} \frac{\widehat\beta_{i+1}}{\widehat\alpha_{i+1}},
  \quad j < k - 1,
\label{eq:wfajk}
\end{equation}
and
\begin{equation}\label{eq:wfajk2}
\widehat{a}_{k-1,k} = \widehat\delta_k + \frac{\widehat\beta_k}{\widehat\alpha_{k-1}} .
\end{equation}
Thus to complete the proof we need to show that $\widehat{a}_{jk}$ defined via (\ref{eq:wfajk}) and (\ref{eq:wfajk2}) coincides with (\ref{eq:ajkwant}) and (\ref{eq:ajkwant2}). Let us begin with the case $j < k - 1$:
\begin{align*}
    \widehat a_{jk} &= \frac{\widehat\alpha_k}{\widehat\alpha_{j+2}}
    \left[
        \left( \frac{\widehat\delta_{j+1}}{\widehat\alpha_{j+1}}
            + \frac{\widehat\beta_{j+1}}{\widehat\alpha_j\widehat\alpha_{j+1}}
        \right)\widehat\beta_{j+2} + \widehat\gamma_{j+2}
    \right]
    \cdot
    \prod_{i = j+2}^{k-1} \frac{\widehat\beta_{i+1}}{\widehat\alpha_{i+1}}
\displaybreak[0] \\ &=
    \frac{\delta_{n-k}}{\delta_{n-j-2}} \Bigg[
        \left(
           -\frac{\delta_{n-j-1}}{\delta_{n-j-1} \delta_{n-j}}
           (\theta_{n-j} + \frac{\alpha_{n-j+1}\beta_{n-j}}{\beta_{n-j+1}})
           + 
            \frac{\delta_{n-j-1}\alpha_{n-j+1}\beta_{n-j}}
            {\delta_{n-j}\delta_{n-j-1}\beta_{n-j+1}}
        \right) \cdot \frac{\delta_{n-j-2}\alpha_{n-j}\beta_{n-j-1}}{\beta_{n-j}} \\
        &\qquad +\frac{\delta_{n-j-2}\beta_{n-j-1}}{\delta_{n-j}\beta_{n-j}}
         (\theta_{n-j}\alpha_{n-j} - \beta_{n-j}\gamma_{n-j}) \Bigg]
         \cdot \prod_{i=j+2}^{k-1}\frac{\delta_{n-i-1}\alpha_{n-i+1}\beta_{n-i}}{\delta_{n-i-1}\beta_{n-i+1}}
\displaybreak[0] \\ &=
    \frac{\delta_{n-k}}{\delta_{n-j-2}} \Bigg[
        \left(
           -\frac{1}{\delta_{n-j}}
           (\theta_{n-j} + \frac{\alpha_{n-j+1}\beta_{n-j}}{\beta_{n-j+1}})
           + 
            \frac{\alpha_{n-j+1}\beta_{n-j}}
            {\delta_{n-j}\beta_{n-j+1}}
        \right) \cdot \frac{\delta_{n-j-2}\alpha_{n-j}\beta_{n-j-1}}{\beta_{n-j}} \\
        &\qquad +\frac{\delta_{n-j-2}\beta_{n-j-1}}{\delta_{n-j}\beta_{n-j}}
         (\theta_{n-j}\alpha_{n-j} - \beta_{n-j}\gamma_{n-j}) \Bigg]
         \cdot \left( \prod_{i=j+2}^{k-1} \alpha_{n-i+1} \right)
         \cdot \left( \prod_{i=j+2}^{k-1} \frac{\beta_{n-i}}{\beta_{n-i+1}} \right)
\displaybreak[0]\\ &=
    \frac{\delta_{n-k}}{\delta_{n-j-2}} \Bigg[
           -\frac{\theta_{n-j}}{\delta_{n-j}}
         \cdot \frac{\delta_{n-j-2}\alpha_{n-j}\beta_{n-j-1}}{\beta_{n-j}} \\
        &\qquad +\frac{\delta_{n-j-2}\beta_{n-j-1}}{\delta_{n-j}\beta_{n-j}}
         (\theta_{n-j}\alpha_{n-j} - \beta_{n-j}\gamma_{n-j}) \Bigg]
         \cdot \left( \prod_{i=j+2}^{k-1}\alpha_{n-i+1} \right) \cdot \frac{\beta_{n-k+1}}{\beta_{n-j-1}}
\displaybreak[0]\\ &=
    \frac{\delta_{n-k}}{\delta_{n-j-2}} \Bigg[
           -\frac{\delta_{n-j-2}\beta_{n-j-1}}{\delta_{n-j}\beta_{n-j}}
         \beta_{n-j}\gamma_{n-j} \Bigg]
         \cdot \left( \prod_{i=j+2}^{k-1}\alpha_{n-i+1} \right) \cdot \frac{\beta_{n-k+1}}{\beta_{n-j-1}}
\displaybreak[0]\\ &=  -\frac{\gamma_{n-j}}{\delta_{n-j}} \cdot
    \left(\prod_{i=j+2}^{k-1} \alpha_{n-i+1} \right) \cdot
    \beta_{n-k+1}\delta_{n-k},
\end{align*}
which is exactly (\ref{eq:ajkwant}). Now, for $j = k - 1$,
\begin{align*}
\widehat a_{k-1,k} = \widehat \delta_k + \frac{\widehat \beta_k}{\widehat\alpha_{k-1}}
    =  -\frac{\delta_{n-k}}{\delta_{n-k+1}}\left(\theta_{n-k+1} +
                \alpha_{n-k+2}\frac{\beta_{n-k+1}}{\beta_{n-k+2}}\right)
    + \frac{\delta_{n-k}\alpha_{n-k+2}\beta_{n-k+1}}{\beta_{n-k+2}\delta_{n-k+1}}.
\end{align*}
After simplification, we get (\ref{eq:ajkwant2}).
\end{proof}


\section{Basis transformation matrices} \label{sec:basis}
For any polynomial basis $Q$ we introduce a basis transformation matrix $S_{PQ}$
for passing from $Q$ to the monomial basis $P = \{1, x, \ldots, x^{n-1}\}$.
The $j$-th column of $S_{PQ}$ contains the coefficients of $Q_{j-1}(x)$, i.e.
\[
    Q_{j-1}(x) = \sum_{i=1}^{n} [S_{PQ}]_{ij} \cdot x^{i-1}, \quad j=1,\ldots,n.
\]
Since $\deg Q_k = k$, the matrix $S_{PQ}$ is upper triangular.
In this section we provide efficient recurrent formulas for columns
of $S_{PQ}$ for cases of quasiseparable, semiseparable and well-free polynomials.
\newline\newline
Recurrence relations for quasiseparable and semiseparable polynomials
(equations (\ref{eq:ego}) and (\ref{eq:semi}), respectively) also have a notion
of auxiliary system of polynomials $G$. These systems, generally speaking, are
not bases. However, given such a system we can define a similar matrix $S_{PG}$
whose columns contain coefficients of polynomials $\{G_k\}_{k=0}^{n-1}$.
Let $\mathbf{s}_k$ and $\mathbf{t}_k$ denote the $k$-th column of
$S_{PQ}$ and $S_{PG}$, respectively.

\begin{lemma} \label{lemma:sktk}
If a system of quasiseparable polynomials $\{Q\}$ satisfies
recurrence relations (\ref{eq:ego}),
then the following recurrence relations hold for columns of $S_{PQ}$ for all $k=1,\ldots,n-1$:
\be
\left[\begin{array}{c}
\mathbf{t}_{k+1} \\
\mathbf{s}_{k+1} 
\end{array}\right]=\left[\begin{array}{cc}
\alpha_k & \beta_k\\
\gamma_k  & (\delta_k Z_0 + \theta_k I)
\end{array}\right]
\left[\begin{array}{c}
\mathbf{t}_k\\
\mathbf{s}_k
\end{array}\right],
\label{skrecur}
\ee
where $Z_0$ is the lower shift matrix. As a consequence, all entries of $S_{PQ}$
can be computed in $\mathcal{O}(n^2)$ operations.
\end{lemma}
\begin{proof}
We can rewrite the matrix multiplication in (\ref{eq:ego}) row by row as
\begin{eqnarray*}
G_k(x) &=& \alpha_k G_{k-1}(x) + \beta_k Q_{k-1}(x), \\
Q_k(x) &=& \gamma_k G_{k-1}(x) + (\delta_k x + \theta_k) Q_{k-1}(x),
\end{eqnarray*}
Addition of polynomials
corresponds to addition of the vectors of their coefficients.
Multiplication of a polynomial by a scalar corresponds to
multiplication of the vector of its coefficients by the scalar.
Multiplication of a polynomial by $x$ corresponds to the lower shift
of the coefficients vector. Since $\mathbf{s}_k$ and $\mathbf{t}_k$
are coefficient vectors of $Q_{k-1}(x)$ and $G_{k-1}(x)$, respectively, system
(\ref{skrecur}) follows.
\end{proof}
\noindent The following shows how to compute the coefficients of the basis transformation matrix passing from the semiseparable basis to the monomial basis. 
\begin{lemma} \label{lemma:sktksemi}
If a system of semiseparable polynomials $\{Q\}$ satisfies
recurrence relations (\ref{eq:semi}),
then the following recurrence relations hold for columns of $S_{PQ}$
for all $k=1,\ldots,n-1$:
\be
\left[\begin{array}{c}
\mathbf{t}_{k+1} \\
\mathbf{s}_{k+1} 
\end{array}\right]=\left[\begin{array}{cc}
\alpha_k & \beta_k\\
\gamma_k  & 1
\end{array}\right]
\left[\begin{array}{c}
\mathbf{t}_k\\
(\delta_k Z_0 + \theta_k I) \mathbf{s}_k
\end{array}\right],
\label{skrecursemi}
\ee
where $Z_0$ is the lower shift matrix. As a consequence, all entries of $S_{PQ}$
can be computed in $\mathcal{O}(n^2)$ operations. 
\end{lemma}
\begin{proof}
Similar to the previous case, we can write the matrix multiplication in (\ref{eq:semi}) row by row as
\begin{eqnarray*}
G_k(x) &=& \alpha_k G_{k-1}(x) + \beta_k (\delta_k x + \theta_k) Q_{k-1}(x), \\
Q_k(x) &=& \gamma_k G_{k-1}(x) + (\delta_k x + \theta_k) Q_{k-1}(x),
\end{eqnarray*}
Again, addition of polynomials
corresponds to addition of the vectors of their coefficients.
Multiplication of a polynomial by a scalar corresponds to
multiplication of the vector of its coefficients by the scalar.
Multiplication of a polynomial by $x$ corresponds to the lower shift
of the coefficients vector. Hence the result (\ref{skrecursemi}) follows.
\end{proof}
\noindent The following shows how to compute column-wise basis transformation matrix from the well-free basis to the monomial basis.   
\begin{lemma} \label{lemma:skwf}
If a system of well-free polynomials $\{ Q\}$ satisfies
recurrence relations (\ref{ewfrr}),
then the following recurrence relations hold for columns of $S_{PQ}$ for all $k=1,\ldots,n-1$:
\begin{eqnarray}
\mathbf{s}_{2} &=& ( \alpha_1 Z_0 - \delta_1 I) \mathbf{s}_{1}
                   ,
\nonumber \\
\mathbf{s}_{k+1} &=& ( \alpha_k Z_0 - \delta_k I) \mathbf{s}_{k}
                   - ( \beta_k Z_0 +  \gamma_k I) \mathbf{s}_{k-1},  \:\: k \geq 2
\label{skrecurwf}
\end{eqnarray}
where $Z_0$ is the lower shift matrix. As a consequence, all entries of $S_{P Q}$
can be computed in $\mathcal{O}(n^2)$ operations.
\end{lemma}
\begin{proof}
Following \cite{BEGKO09}, if the system of well-free polynomials satisfies recurrence relation (\ref{ewfrr}), then the confederate matrix has the form  
\[
C(Q_n)=\begin{bmatrix}
\frac{\delta_1 }{\alpha_1 } & \frac{\frac{\delta_1}{\alpha_1}\beta_2+\gamma_2  }{\alpha_2} & \frac{\frac{\delta_1}{\alpha_1}\beta_2+\gamma_2}{\alpha_2}\left (\frac{\beta_3}{\alpha_3}  \right ) &\cdots  & \cdots & \frac{\frac{\delta_1}{\alpha_1}\beta_2+\gamma_2}{\alpha_2}\left ( \frac{\beta}{\alpha} \right )^{\times}_{2,n+1}\\ 
 \frac{1}{\alpha_1}&  \frac{\delta_2  }{\alpha_2}+\frac{\beta_2  }{\alpha_1\alpha_2}&\frac{\left ( \frac{\delta_2  }{\alpha_2}+\frac{\beta_2  }{\alpha_1\alpha_2} \right )\beta_3+\gamma_3  }{\alpha_3}  & \cdots &\cdots & \frac{\left (  \frac{\delta_2  }{\alpha_2}+\frac{\beta_2  }{\alpha_1\alpha_2}\right )\beta_3+\gamma_3}{\alpha_3} \left ( \frac{\beta}{\alpha} \right )^{\times}_{3,n+1}\\ 
 0& \frac{1}{\alpha_2} & \frac{\delta_3  }{\alpha_3}+ \frac{\beta_3  }{\alpha_2\alpha_3}&  &  & \vdots\\ 
 0&0  & \ddots & \ddots &  & \vdots\\ 
 \vdots& \vdots &  \ddots& \ddots & \ddots & \\ 
0 &0  &  \cdots& 0&\frac{1}{\alpha_{n-1}}   & \frac{\delta_n  }{\alpha_n}+\frac{\beta_n  }{\alpha_{n-1}\alpha_n}
\end{bmatrix}.
\]
Thus we can see $C(Q_n)$ as 
\be
C(Q_n)=\mathit{L}+\mathit{D}+Z_0^T\mathit{D_1}\mathit{U}
\label{eq:cwfb}
\ee 
where
\[
\begin{matrix}
\mathit{L}=\begin{bmatrix}
0 & 0 &  \cdots&  & 0\\ 
 \frac{1}{\alpha_1}& 0 &  \cdots&  & 0\\ 
0 & \frac{1}{\alpha_2} &0  &\cdots  & 0\\ 
 \vdots&  & \ddots &\ddots  &\vdots \\ 
0 &  \cdots& 0 &  \frac{1}{\alpha_{n-1}}& 0
\end{bmatrix}
&
\mathit{D}=\begin{bmatrix}
\frac{\delta_1 }{\alpha_1 } & 0 &  \cdots&  & 0\\ 
 0& \frac{\delta_2 }{\alpha_2 }+\frac{\beta_2}{\alpha_1\alpha_2} & 0 &\cdots  & 0\\ 
 &  \ddots& \ddots &  &\vdots\\ 
 \vdots&  & \ddots &\ddots  & 0\\ 
0 &  \cdots& 0 &  0&\frac{\delta_n  }{\alpha_n}+\frac{\beta_n  }{\alpha_{n-1}\alpha_n}
\end{bmatrix}
\end{matrix}
\]
\[
\begin{matrix}
\mathit{D_1}=\begin{bmatrix}
1 & 0 &  \cdots&  & 0\\ 
 0&\frac{\frac{\delta_1}{\alpha_1}\beta_2+\gamma_2}{\alpha_2}  & 0 &\cdots  & 0\\ 
0 & 0& \frac{\left (\frac{\delta_2}{\alpha_2}+\frac{\beta_2}{\alpha_1\alpha_2}  \right )\beta_3+\gamma_3}{\alpha_3} &  &\vdots\\ 
 \vdots&  & \ddots &\ddots  & 0\\ 
0 &  \cdots& 0 &  0&\frac{\left (\frac{\delta_{n-1}}{\alpha_{n-1}}+\frac{\beta_{n-1}}{\alpha_{n-2}\alpha_{n-1}}  \right )\beta_n+\gamma_n}{\alpha_n}
\end{bmatrix}
\end{matrix}
\]
\[
\mathit{U}=\begin{bmatrix}
1 &0  & 0 & \cdots &  &0 \\ 
 0&1  &\frac{\beta_3}{\alpha_3 }  & \frac{\beta_3}{\alpha_3 }\left (\frac{\beta_4}{\alpha_4 }  \right )   & \cdots & \frac{\beta_3}{\alpha_3 }\left (\frac{\beta_4}{\alpha_4 }  \right )\cdots\left (\frac{\beta_n}{\alpha_n }  \right )\\ 
 0&0  & 1 & \frac{\beta_4}{\alpha_4 }   & \cdots &\frac{\beta_4}{\alpha_4 }  \cdots\left (\frac{\beta_n}{\alpha_n }  \right ) \\ 
\vdots &  &  & \ddots &\ddots  & \vdots\\ 
 0&  & \cdots & 0 &  1& \frac{\beta_n}{\alpha_n }  \\ 
 0&  &  \cdots&  &  0& 1
\end{bmatrix}.
\]
Since $\mathit{U}^{-1}$ is bidiagonal (by \cite{BEGKO09}), the system $S_{PQ}\:C(Q_n)=Z_0\:S_{PQ}$ (by \cite{MB79})  together with (\ref{eq:cwfb}) can be seen as
\be
\begin{matrix}
S_{PQ}\:(\mathit{L}\mathit{U}^{-1}+\mathit{D}\mathit{U}^{-1}
+Z_0^T\mathit{D_1})=Z_0\:S_{PQ}\mathit{U}^{-1}
\end{matrix}
\label{eq:bandwf}
\ee
where 
\begin{center}
$\mathit{U}^{-1}=\begin{bmatrix}
1 & 0 & 0 & \cdots &  &0 \\ 
0 &1  & -\frac{\beta_3}{\alpha_3}  &  0&  \cdots& 0\\ 
 &  &  &  &  & \vdots\\ 
 \vdots&  &  &  \ddots& \ddots &0 \\ 
0 &  &  \cdots&  0& 1 &-\frac{\beta_n}{\alpha_n}  \\ 
0 &  & \cdots &  &  0&1 
\end{bmatrix}.
$
\end{center}
Hence the explicit matrix form of the system (\ref{eq:bandwf}) can be rewritten  as 
\be
\begin{matrix}
\begin{bmatrix}
s_{11} & s_{12} & s_{13} &  \cdots&  &s_{1\:n} \\ 
0 & s_{22} & s_{23} &  &  & s_{2\:n}\\ 
 & \ddots &  & \ddots &  \ddots& \vdots\\ 
\vdots &  &  & \ddots &  & \\ 
 &  &  &  &  & s_{n-1\:n}\\ 
 0&  &\cdots  &  &  0&s_{n\:n} 
\end{bmatrix}
\begin{bmatrix}
\frac{\delta_1 }{\alpha_1 } &  \frac{\frac{\delta_1}{\alpha_1}\beta_2+\gamma_2}{\alpha_2} & 0 & 0 &  \cdots&  & 0\\ 
 \frac{1}{\alpha_1 }&\frac{\delta_2}{\alpha_2}+\frac{\beta_2}{\alpha_1\alpha_2} & \frac{\gamma_3  }{\alpha_3 } & 0 &  \cdots&  & 0\\ 
0 &\frac{1}{\alpha_2 }  & \frac{\delta_3 }{\alpha_3 } &  \frac{\gamma_4  }{\alpha_4 }& 0 & \cdots & 0\\ 
0 &  & \ddots & \ddots &\ddots  &\ddots  & \vdots\\ 
 \vdots&  & \ddots &  \ddots& \ddots &  &0 \\ 
 0&  & \cdots &  0&\frac{1 }{\alpha_{n-2} }   & \frac{\delta_{n-1} }{\alpha_{n-1} } &\frac{\gamma_{n} }{\alpha_{n} } \\ 
 0&  & \cdots &  & 0 & \frac{1 }{\alpha_{n-1} } & \frac{\delta_{n} }{\alpha_{n} }
\end{bmatrix}
\\
= \begin{bmatrix}
0 & 0 & 0 &  \cdots&  &0 \\ 
1 & 0 & 0 &  &  & 0\\ 
 0& 1 &  & \ddots &  & \vdots\\ 
\vdots &  &  & \ddots &  & \\ 
 &  & \ddots & \ddots &  & 0\\ 
 0&  &\cdots  & 0 &  1&0 
\end{bmatrix}
\begin{bmatrix}
s_{11} & s_{12} & s_{13} &  \cdots&  &s_{1\:n} \\ 
0 & s_{22} & s_{23} &  &  & s_{2\:n}\\ 
 & \ddots &  & \ddots &  \ddots& \vdots\\ 
\vdots &  &  &\ddots  &  & \\ 
 &  &  &  &  & s_{n-1\:n}\\ 
 0&  &\cdots  &  &  0&s_{n\:n} 
\end{bmatrix}
\begin{bmatrix}
1 & 0 & 0 & \cdots &  &0 \\ 
0 &1  & -\frac{\beta_3}{\alpha_3}  &  0&  \cdots& 0\\ 
 &  &  &  &  & \vdots\\ 
 \vdots&  &  &  \ddots& \ddots &0 \\ 
0 &  &  \cdots&  0& 1 &-\frac{\beta_n}{\alpha_n}  \\ 
0 &  & \cdots &  &  0&1 
\end{bmatrix}.
\end{matrix}
\label{eq:mbandwf}
\ee 
Notice that $(\mathit{L}\mathit{U}^{-1}+\mathit{D}\mathit{U}^{-1}
+Z_0^T\mathit{D_1})$ in (\ref{eq:bandwf}) has reduced to a tridiagonal matrix. 
\\
Now by multiplying the system (\ref{eq:mbandwf}) by $e_1^T$ from right, we have
\begin{center}
$
\begin{bmatrix}
\frac{\delta_1}{\alpha_1}s_{11}+\frac{1}{\alpha_1}s_{12}\\ 
\frac{1}{\alpha_1}s_{22}\\ 
0\\ 
\vdots
\\
0
\end{bmatrix}=\begin{bmatrix}
0\\ 
s_{11}\\ 
\vdots\\ 
\\ 
0
\end{bmatrix}
$
\end{center}
and rearranging
\[
{\bf s_2} = \alpha_1 Z_0 {\bf s_1} - \delta_1 I{\bf s_1}
\]
gives the second column. Next by multiplying the system (\ref{eq:mbandwf}) by $e_2^T$ from right, we have
\begin{center}
$
\begin{bmatrix}
\left( \frac{\frac{\delta_1}{\alpha_1}\beta_2+\gamma_2}{\alpha_2}\right)s_{11}+\left(\frac{\delta_2}{\alpha_2}+\frac{\beta_2}{\alpha_1\alpha_2}\right)s_{12}+\frac{1}{\alpha_2}s_{13}\\ 
\left(\frac{\delta_2}{\alpha_2}+\frac{\beta_2}{\alpha_1\alpha_2}\right)s_{22}+\frac{1}{\alpha_2}s_{23}\\ 
\frac{1}{\alpha_2}s_{33}\\
0\\ 
\vdots
\\
0
\end{bmatrix}=\begin{bmatrix}
0\\ 
s_{12}\\ 
s_{22}\\
0\\ 
\vdots\\ 
0
\end{bmatrix}
$
\end{center}
and rearranging
\[
{\bf s_3}=\alpha_2 \:Z_0 {\bf s_2}-\delta_2 I \:{\bf s_2}- \beta_2 Z_0 {\bf s_1}- \gamma_2 I \:{\bf s_1}
\] 
gives the result for $k=2$. Next multiplying the system (\ref{eq:mbandwf}) by $e_3^T$ from right, we have
\begin{center}
$
\begin{bmatrix}
\left ( \frac{\gamma_3 }{\alpha_3 } \right )s_{12}+\left ( \frac{\delta_3}{\alpha_3 } \right )s_{13}+\frac{1}{\alpha_3}s_{14}\\ 
\left ( \frac{\gamma_3 }{\alpha_3 } \right )s_{22}+\left ( \frac{\delta_3}{\alpha_3 } \right )s_{23}+\frac{1}{\alpha_3}s_{24}\\ 
\left ( \frac{\delta_3}{\alpha_3 } \right )s_{33}+\frac{1}{\alpha_3}s_{34}\\
\frac{1}{\alpha_3}s_{44}\\ 
0\\
\vdots
\\
0
\end{bmatrix}=\begin{bmatrix}
0\\ 
-\frac{\beta_3}{\alpha_3}s_{12}+s_{13}\\ 
-\frac{\beta_3}{\alpha_3}s_{22}+s_{23}\\
s_{33}\\
0\\ 
\vdots\\ 
0
\end{bmatrix}
$
\end{center}
and rearranging
\[
{\bf s_4}=\alpha_3 \:Z_0{\bf s_3}-\delta_3 I {\bf s_3}-\beta_3\:Z_0{\bf s_2}-
\gamma_3 I {\bf s_2}
\] 
gives the result for $k=3$. Continuing in this fashion, we can recover all columns in the basis transformation matrix and hence the result (\ref{skrecurwf}) follows.
\end{proof}


\section{Computation of sums in the inversion formula}
\label{sec:sums}
In this section we analyze the cost of computing the sum $V_Q(x)\cdot(\sum_{k=1}^n d_k W_Q^{k-1})$ in the inversion formula (\ref{eq:invformula}) of polynomial Vandermonde-like matrices. We continue the discussion by covering the cost of computing $V_Q(x)\cdot(\sum_{k=1}^n d_k W_Q^{k-1})$ for quasiseparable, semiseparable and well-free polynomials $Q$. Before computing the cost of the summation corresponding to the quasiseparable, semiseparable and well-free polynomials let us state a supporting result from \cite{KO97} which enables us to continue the discussion. 
\begin{lemma}\label{lemma:vqvf}
Let $\{Q\}$ be a system of polynomials,
$V_Q(x)$ be the polynomial Vandermonde matrix for $Q$ with nodes $x_1, x_2, \ldots, x_n$,
let $S_{PQ}$ be the matrix corresponding to passing from $Q$ to the monomial
basis $P$ and let numbers
$d_1, d_2, \ldots, d_n$ be arbitrary. Then the following displacement
equation holds:
\begin{equation}
V_Q(x)\cdot\left(\sum_{k=1}^n d_k W_Q^{k-1}\right)
=
\left(\sum_{k=1}^n d_k D_{\frac{1}{x}}^{k-1}\right)
 \cdot V_Q(x) - V_F(1/x) \cdot S_{PQ},
\end{equation}
where
\begin{equation}
V_F(1/x) = \left[\begin{array}{cccc}
F_0(\frac{1}{x_1}) & 
F_1(\frac{1}{x_1}) & 
\cdots &
F_{n-1}(\frac{1}{x_1})
\\
F_0(\frac{1}{x_2}) & 
F_1(\frac{1}{x_2}) & 
\cdots &
F_{n-1}(\frac{1}{x_2})
\\
\vdots &
\vdots &
       &
\vdots
\\
F_0(\frac{1}{x_n}) & 
F_1(\frac{1}{x_n}) & 
\cdots &
F_{n-1}(\frac{1}{x_n})
\end{array}\right],
\label{vfoneoverx}
\end{equation}
with
\begin{eqnarray}
F_{n-1}(1/x) & = & 0, \nonumber \\
F_k(1/x) & = & \frac{1}{x} \cdot (F_{k+1}(1/x) + d_{k+2}), \quad (k = 0,1,\ldots,n-2). \label{fkrecur}
\end{eqnarray}
\end{lemma}
\begin{proof}
See Lemma 7.2 in\cite{KO97}.
\end{proof}
\noindent Lemma \ref{lemma:vqvf} shows that the problem of computing the elements
of $V_Q(x)\cdot(\sum_{k=1}^n d_k W_Q^{k-1})$ can be reduced to computing
the matrices $(\sum_{k=1}^n d_k D_{\frac{1}{x}}^{k-1}) \cdot V_Q(x)$
and $V_F(1/x) \cdot S_{PQ}$.
To obtain the entries of the former matrix, we note that for quasiseparable,
semiseparable and well-free polynomials the entries
of $V_Q(x)$ can be computed in $\mathcal{O}(n^2)$. The matrix $\sum_{k=1}^n d_k D_{\frac{1}{x}}^{k-1}$
is diagonal and can be computed in $\mathcal{O}(n^2)$. Finally, multiplication
of a dense matrix $V_Q(x)$ by a diagonal matrix can be done in $\mathcal{O}(n^2)$ operations.
Therefore, the entries of the matrix
$(\sum_{k=1}^n d_k D_{\frac{1}{x}}^{k-1}) \cdot V_Q(x)$
can be computed in $\mathcal{O}(n^2)$ operations.
\\\\
To get an $\mathcal{O}(n^2)$ algorithm for computation of 
$V_Q(x)\cdot(\sum_{k=1}^n d_k W_Q^{k-1})$, it remains
to show that the entries of $V_F(1/x)\cdot S_{PQ}$
can be computed in $\mathcal{O}(n^2)$ operations as well.
To do so, we need to exploit the special structure of matrices $V_F(1/x)$
and $S_{PQ}$. 
\begin{lemma} \label{lemma:vfdisp}
The matrix $V_F(1/x)$ defined by (\ref{vfoneoverx})
satisfies the following displacement equation:
\begin{equation}\label{eq:vfdisp}
V_F(1/x) \cdot Z_0 =
D_x \cdot V_F(1/x) - 
\left[\begin{array}{c}
1 \\
1 \\
\vdots \\
1 \\
 \end{array}\right]
\cdot
\left[\begin{array}{ccccc}
d_2 &
d_3 &
\cdots &
d_n &
0
 \end{array}\right],
\end{equation}
where $Z_0$ is the lower shift matrix.
\end{lemma}
\begin{proof}
This can easily be checked by matrix multiplication.
\end{proof}
\noindent Now we are ready to show that in cases of quasiseparable,
semiseparable and well-free polynomials, the entries
of $V_Q(x)\cdot(\sum_{k=1}^n d_k W_Q^{k-1})$ can be computed in
$\mathcal{O}(n^2)$ operations. To initiate let's start the computation for quaiseparable polynomials $Q$ satisfying the recurrence relations (\ref{eq:ego}).

\begin{lemma}\label{lemma:qs72}
Let a system of quasiseparable polynomials $\{Q\}$ satisfy
the recurrence relations (\ref{eq:ego}),
the matrices $V_Q(x)$ and $W_Q$ be defined as in (\ref{eq:vq}) and (\ref{eq:wq}),
and let numbers
$d_1, d_2, \ldots, d_n$ be arbitrary. Then the complexity of computing
the entries of the matrix $V_Q(x)\cdot(\sum_{k=1}^n d_k W_Q^{k-1})$
is no more than $\mathcal{O}(n^2)$ operations.
\begin{proof}
By Lemma \ref{lemma:vqvf}, we only need to show that
the entries of $V_F(1/x)\cdot S_{PQ}$ can be computed in $\mathcal{O}(n^2)$
operations. By Lemma \ref{lemma:sktk}, columns of $S_{PQ}$ satisfy
recurrence relations (\ref{skrecur}).
After multiplying each equation in the system
(\ref{skrecur}) by $V_F(1/x)$, we get recurrence relations
for columns $V_F(1/x)\mathbf{s}_{k+1}$ of product $V_F(1/x)\cdot S_{PQ}$:
\begin{eqnarray}
V_F(1/x)\mathbf{t}_{k+1} &=& \alpha_k V_F(1/x) \mathbf{t}_{k}
 + \beta_k V_F(1/x) \mathbf{s}_{k},
\label{vtkrecur}
\\
V_F(1/x) \mathbf{s}_{k+1} &=& \gamma_k V_F(1/x) \mathbf{t}_{k}
  + \delta_k V_F(1/x) Z_0 \mathbf{s}_{k} + \theta_k V_F(1/x) \mathbf{s}_{k}.
\label{vskrecur}
\end{eqnarray}
Because of the term $\delta_kV_F(1/x)Z_0\mathbf{s}_k$,
these recurrence relations do not provide
an efficient way to compute columns of $V_F(1/x)\cdot S_{PQ}$.
Using Lemma \ref{lemma:vfdisp}, we can substitute (\ref{eq:vfdisp})
into equation (\ref{vskrecur}) to get
\begin{equation}
\label{vskrecurnew}
\begin{split}
V_F(1/x) \mathbf{s}_{k+1} &= \gamma_k V_F(1/x) \mathbf{t}_{k}
  + \delta_k D_x \cdot (V_F(1/x) \mathbf{s}_{k})
  \\ &- \delta_k
\left[\begin{array}{c}
1 \\
1 \\
\vdots \\
1 \\
 \end{array}\right]
\cdot
\left[\begin{array}{ccccc}
d_2 &
d_3 &
\cdots &
d_n &
0
 \end{array}\right] \mathbf{s}_k
 \\ &+ \theta_k V_F(1/x) \mathbf{s}_{k}.
\end{split}
\end{equation}
Given $V_F(1/x)\mathbf{s}_k$ and $V_F(1/x)\mathbf{t}_k$, formulas
(\ref{vtkrecur}) and (\ref{vskrecurnew})
allow us to compute $V_F(1/x) \mathbf{s}_{k+1}$
and $V_F(1/x) \mathbf{t}_{k+1}$ in $\mathcal{O}(n)$ operations.
Hence all $n$ columns of matrix $V_F(1/x)\cdot S_{PQ}$
can be computed in $\mathcal{O}(n^2)$.
\end{proof}
\end{lemma}
\noindent The next result shows the cost of the computation $V_Q(x)\cdot(\sum_{k=1}^n d_k W_Q^{k-1})$ for semiseparable polynomials $Q$ satisfying recurrence relation (\ref{eq:semi}).
\begin{lemma}\label{lemma:ss72}
Let a system of semiseparable polynomials $\{Q\}$ satisfy
the recurrence relations (\ref{eq:semi}),
the matrices $V_Q(x)$ and $W_Q$ be defined as in (\ref{eq:vq}) and (\ref{eq:wq}),
and let numbers
$d_1, d_2, \ldots, d_n$ be arbitrary. Then the complexity of computing
the entries of the matrix $V_Q(x)\cdot(\sum_{k=1}^n d_k W_Q^{k-1})$
is no more than $\mathcal{O}(n^2)$ operations.
\begin{proof}
By Lemma \ref{lemma:vqvf}, we only need to show that
the entries of $V_F(1/x)\cdot S_{PQ}$ can be computed in $\mathcal{O}(n^2)$
operations. By Lemma \ref{lemma:sktksemi}, columns of $S_{PQ}$ satisfy
recurrence relations (\ref{skrecursemi}).
After multiplying each equation in the system (\ref{skrecursemi}) by $V_F(1/x)$, we get recurrence relations
for columns $V_F(1/x)\mathbf{s}_{k+1}$ of product $V_F(1/x)\cdot S_{PQ}$:
\begin{eqnarray}
V_F(1/x)\mathbf{t}_{k+1} &=& \alpha_k V_F(1/x) \mathbf{t}_k
                  + \beta_k \delta_k V_F(1/x) Z_0 \mathbf{s}_k
                  + \beta_k \theta_k V_F(1/x) \mathbf{s}_k,
\label{vtkrecursemi}
\\
V_F(1/x)\mathbf{s}_{k+1} &=& \gamma_k V_F(1/x) \mathbf{t}_k
        + \delta_k V_F(1/x) Z_0 \mathbf{s_k}  + \theta_k V_F(1/x) \mathbf{s}_k.
\label{vskrecursemi}
\end{eqnarray}
By Lemma \ref{lemma:vfdisp}, we can substitute (\ref{eq:vfdisp})
into equations (\ref{vtkrecursemi}) and (\ref{vskrecursemi}) to get
\begin{equation}
\label{vtkrecurseminew}
\begin{split}
V_F(1/x) \mathbf{t}_{k+1} &= \alpha_k V_F(1/x) \mathbf{t}_{k}
  + \beta_k \delta_k D_x \cdot (V_F(1/x) \mathbf{s}_{k})
  \\ &- \beta_k \delta_k
\left[\begin{array}{c}
1 \\
\vdots \\
1 \\
 \end{array}\right]
\cdot
\left[\begin{array}{ccccc}
d_2 &
d_3 &
\cdots &
d_n &
0
 \end{array}\right] \mathbf{s}_k
 \\ &+ \beta_k \theta_k V_F(1/x) \mathbf{s}_{k}.
\end{split}
\end{equation}

\begin{equation}
\label{vskrecurseminew}
\begin{split}
V_F(1/x) \mathbf{s}_{k+1} &= \gamma_k V_F(1/x) \mathbf{t}_{k}
  + \delta_k D_x \cdot (V_F(1/x) \mathbf{s}_{k})
  \\ &- \delta_k
\left[\begin{array}{c}
1 \\
\vdots \\
1 \\
 \end{array}\right]
\cdot
\left[\begin{array}{ccccc}
d_2 &
d_3 &
\cdots &
d_n &
0
 \end{array}\right] \mathbf{s}_k
 \\ &+ \theta_k V_F(1/x) \mathbf{s}_{k}.
\end{split}
\end{equation}
Given $V_F(1/x)\mathbf{s}_k$ and $V_F(1/x)\mathbf{t}_k$, formulas
(\ref{vtkrecurseminew}) and (\ref{vskrecurseminew})
allow one to compute $V_F(1/x) \mathbf{s}_{k+1}$
and $V_F(1/x) \mathbf{t}_{k+1}$ in $\mathcal{O}(n)$ operations.
Hence all $n$ columns of matrix $V_F(1/x)\cdot S_{PQ}$
can be computed in $\mathcal{O}(n^2)$.
\end{proof}
\end{lemma}
\noindent The following result shows the cost of the computation $V_Q(x)\cdot(\sum_{k=1}^n d_k W_Q^{k-1})$ for well-free polynomials $Q$ satisfying recurrence relation (\ref{ewfrr}).
\begin{lemma}\label{lemma:wf72}
Let a system of well-free polynomials $Q$ satisfy
the recurrence relations (\ref{ewfrr}),
the matrices $V_Q(x)$ and $W_Q$ be defined as in (\ref{eq:vq}) and (\ref{eq:wq}),
and let numbers
$d_1, d_2, \ldots, d_n$ be arbitrary. Then the complexity of computing
the entries of the matrix $V_Q(x)\cdot(\sum_{k=1}^n d_k W_Q^{k-1})$
is no more than $\mathcal{O}(n^2)$ operations.
\begin{proof}
By Lemma \ref{lemma:vqvf}, we only need to show that
the entries of $V_F(1/x)\cdot S_{PQ}$ can be computed in $\mathcal{O}(n^2)$
operations. By Lemma \ref{lemma:skwf}, columns of $S_{PQ}$ satisfy
recurrence relation (\ref{skrecurwf}).
After multiplying (\ref{skrecurwf}) by $V_F(1/x)$, we get recurrence relations
for columns $V_F(1/x)\mathbf{s}_{k+1}$ of product $V_F(1/x)\cdot S_{PQ}$:

\begin{eqnarray}
V_F(1/x) \mathbf{s}_{2} &=& \alpha_1 V_F(1/x) Z_0 \mathbf{s}_{1}
              - \delta_1 V_F(1/x) \mathbf{s}_{1},  \nonumber \\
V_F(1/x) \mathbf{s}_{k+1} &=& \alpha_k V_F(1/x) Z_0 \mathbf{s}_{k}
                        - \delta_k V_F(1/x) \mathbf{s}_{k} \nonumber \\
                   & &\qquad - \beta_k V_F(1/x) Z_0 \mathbf{s}_{k-1}
                   - \gamma_k V_F(1/x) \mathbf{s}_{k-1},
                    \quad k\geq 2.
\label{vskrecurwf}
\end{eqnarray}

By Lemma \ref{lemma:vfdisp}, we can substitute (\ref{eq:vfdisp})
into equation (\ref{vskrecurwf}) to get
\begin{equation}
\label{vskrecurwfnew}
\begin{split}
V_F(1/x) \mathbf{s}_{2} &=
    \alpha_1 D_x (V_F(1/x) \mathbf{s}_{1}) \\
  &- \alpha_1 
\left[\begin{array}{c} 1 \\ \vdots \\ 1 \\ \end{array}\right]
\cdot
\left[\begin{array}{ccccc} d_2 & d_3 & \cdots & d_n & 0 \end{array}\right]
\cdot
\mathbf{s}_1 \\
  &- \delta_1 V_F(1/x) \mathbf{s}_{1},\\
V_F(1/x) \mathbf{s}_{k+1} &= \alpha_k D_x (V_F(1/x) \mathbf{s}_k) 
    -\delta_k V_F(1/x) \mathbf{s}_k\\
    &- \alpha_k 
\left[\begin{array}{c} 1 \\ \vdots \\ 1 \\ \end{array}\right]
\cdot
\left[\begin{array}{ccccc} d_2 & d_3 & \cdots & d_n & 0 \end{array}\right]
\cdot
\mathbf{s}_k \\
    &- \beta_k D_x ( V_F(1/x) \mathbf{s}_{k-1} ) - \gamma_k V_F(1/x) \mathbf{s}_{k-1} \\
    &+ \beta_k 
\left[\begin{array}{c} 1 \\ \vdots \\ 1 \\ \end{array}\right]
\cdot
\left[\begin{array}{ccccc} d_2 & d_3 & \cdots & d_n & 0 \end{array}\right]
\cdot
\mathbf{s}_{k-1}, \quad k \geq 2.
\end{split}
\end{equation}
Given $V_F(1/x)\mathbf{s}_{k-1}$ and $V_F(1/x)\mathbf{s}_k$, formula
(\ref{vskrecurwfnew}) allows one to compute $V_F(1/x) \mathbf{s}_{k+1}$
in $\mathcal{O}(n)$ operations.
Hence all $n$ columns of matrix $V_F(1/x)\cdot S_{PQ}$
can be computed in $\mathcal{O}(n^2)$.
\end{proof}
\end{lemma}

\section{Inversion algorithm}
\label{sec:invalgo}
In this section we conclude with a formula for inversion of polynomial Vandermonde-like matrices with $\mathcal{O}(n^2)$ complexity. We summarize our results in Algorithm \ref{algo:inv}, which takes as input recurrence relations of the form
(\ref{eq:ego}), (\ref{eq:semi}) or (\ref{ewfrr}),
matrices $G$ and $B$, and numbers $x_1,\ldots,x_n$.
The output is all elements
of the matrix $R^{-1}$, where $R$ is defined by displacement equation
(\ref{eq:dispwq}).
\begin{algorithm}
\caption{Inversion of quasiseparable-Vandermonde-like matrices}
\label{algo:inv}
\begin{algorithmic}[1]
\State Find recurrence relations for $\widehat Q$ using
Lemma \ref{lemma:hatqs}, \ref{lemma:hatss} or \ref{lemma:hatwf}.
\State Use Algorithm \ref{algo:gepp} to solve $2\alpha$ linear systems
with $R$ in order to compute $c_i$ given by (\ref{eq:cidik})
and $G^T R^{-T}$. Use Lemma \ref{lemma:mnqs}, \ref{lemma:mnss} or \ref{lemma:emnwf} 
to get generators of matrices $M_Q-\xi N_Q$.
\State Compute all elements of $S_{P \widehat Q}$ using
Lemma \ref{lemma:sktk}, \ref{lemma:sktksemi} or \ref{lemma:skwf}.
\State Compute $d_{ik}$ defined by (\ref{eq:cidik})
    via solving $\alpha$ linear systems with $S_{P \widehat Q}$ by
    back substitution.
\State Use Lemma \ref{lemma:qs72}, \ref{lemma:ss72} or \ref{lemma:wf72}
to compute 
$(\sum_{k=1}^{n} d_{ik} (W_{\widehat{Q}}^T)^{k-1}) \cdot V_{\widehat{Q}}^T $
for $i = 1,\ldots,\alpha$.
\State Finally, compute $R^{-1}$ using (\ref{eq:invformula}).
\end{algorithmic}
\end{algorithm}
\bt
Let a system of polynomials $\{Q\}$ be defined by recurrence
relations (\ref{eq:ego}), (\ref{eq:semi}) or (\ref{ewfrr}).
Let $G \in \mathbb C ^{n \times \alpha}$,
    $B \in \mathbb C ^{\alpha \times n}$ be arbitrary matrices
and $x_1,\ldots,x_n$ be arbitrary nonzero numbers.
Let $R \in \mathbb C ^{n \times n}$ be defined by displacement
equation (\ref{eq:dispwq}).
Then the complexity of computing all elements of $R^{-1}$
is no more than $\mathcal{O}(\alpha n^2)$ operations.
\et
\begin{proof}
Consider Algorithm \ref{algo:inv}.
Step 1 can be performed in $\mathcal{O}(n)$ operations.
By Theorem \ref{thm:gepp}, step 2 can be done in $\mathcal{O}(\alpha n^2)$
operations.
By Lemma \ref{lemma:sktk}, \ref{lemma:sktksemi} or \ref{lemma:skwf},
step 3 can be done in $\mathcal{O}(n^2)$.
Step 4 can be done in $\mathcal{O}(\alpha n^2)$ operations, since
$S_{P\widehat Q}$ is a triangular matrix.
By Lemma \ref{lemma:qs72}, \ref{lemma:ss72} or \ref{lemma:wf72},
the complexity of step 5 is no more than $\mathcal{O}(\alpha n^2)$ operations.
Step 6 can be done in $\mathcal{O}(\alpha n^2)$ operations as well.
Therefore, the complexity of Algorithm \ref{algo:inv} is 
$\mathcal{O}(\alpha n^2)$ operations.
\end{proof}


\section{Conclusion}
In this paper we introduced a fast algorithm to compute the inversion of quasiseparable Vandermonde-like matrices with the help of a fast Gaussian elimination algorithm for polynomial Vandermonde-like matrices. To derive the former algorithm we mainly identified the structures of displacement operators $W_{Q}$ for quasiseparable, semiseparable, and well-free polynomials in terms of recurrence relations generators for $M_Q$ and $N_Q$ with $\mathcal{O}(n)$ complexity. We also identified the columns of basis transformations matrices satisfying two-term recurrence relations(for quasiseparable and semiseparable polynomials) and three-term recurrence realtions(for well-free polynomials) and hence the cost of computing entries of the basis transformation matrices have $\mathcal{O}(n^2)$ complexity. We also recognized the recurrence relations for the generalized associated polynomials in terms of generators $\alpha, \beta, \gamma$ and $\delta$ corresponding to quasiseparable, semiseparable and well-free polynomials. By combining all of this we were able to simply derive a fast $\mathcal{O}(n^2)$ inversion algorithm to generalize the results of Quasiseparable Vandermonde matrices to a wider class of Quasiseparable Vandermonde-like matrices.


\begin{thebibliography}{YY}


\bibitem{BEGKO07} T.Bella, Y.Eidelman, I.Gohberg, I. Koltracht and V.Olshevsky, {\sl A Bjorck-Pereyra-type algorithm for Szego-Vandermonde matrices based on properties of unitary Hessenberg matrices}, Linear Algebra and Applications, {\bf 420, 2-3} (2007), 634-647. 

\bibitem{BEGKO09} T.Bella, Y.Eidelman, I.Gohberg, I. Koltracht and V.Olshevsky, {\sl A fast Bjorck-Pereyra-type algorithm for solving Hessenberg-quasiseparable-Vandermonde systems}, SIAM. J. Matrix Anal. and Appl., {\bf 31, 2}, (2009),  790-815.

\bibitem{BEGO13} T.Bella, Y.Eidelman, I.Gohberg, I. Koltracht and V.Olshevsky, {\sl Classifications of three-term and two-term recurrence relations via subclasses of quasiseparable matrices}, submitted to SIAM Journal of Matrix Analysis (SIMAX), 2013

\bibitem{BEGOT} T.Bella, Y.Eidelman, I.Gohberg, V.Olshevsky, E.Tyrtyshnikov, {\sl Fast inversion of Hessenberg-quasiseparable-Vandermonde matrices and resulting recurrence relations and characterizations}, preprint.

\bibitem{BEGOT13} T. Bella, Y. Eidelman, I. Gohberg, V. Olshevsky and E. Tyrtyshnikov, \textit{Fast inversion of polynomial-Vandermonde matrices for polynomial systems related to order one quasiseparable matrices}, Advances in Structured Operator Theory and Related Areas, Operator Theory: Advances and Applications, Vol. 237, Pages 79-106 (2013).


\bibitem{BEGOT12} T. Bella, Y. Eidelman, I. Gohberg, V. Olshevsky, E. Tyrtyshnikov,
{\sl Fast Traub--like inversion algorithm for Hessenberg order one quasiseparable
Vandermonde matrices}, submitted to Journal of Complexity, 2012.

\bibitem{BEGOTZ10} T. Bella, Y. Eidelman, I. Gohberg, V. Olshevsky, E. Tyrtyshnikov,
{\sl A Traub-like algorithm for Hessenberg-quasiseparable-Vandermonde matrices of
arbitrary order}, Numerical methods for structured matrices and applications, 127â€“154, Operator
Theory: Advances and Applications, 199, Birkhauser Verlag, Basel (2010)


\bibitem{BOZ11} T. Bella, V. Olshevsky and P. Zhlobich, {\sl Classifications of Recurrence Relations via Subclasses of (H, m)-quasiseparable Matrices}, Numerical Linear Algebra in Signals, Systems and Control: Lecture Notes in Electrical Engineering, 80, Springer, Netherlands (2011).

\bibitem{BP70} A. Bj\:aorck and V. Pereyra, {\sl Solution of Vandermonde Systems of Equations}, Math. Comp., {\bf 24} (1970), 893-903.

\bibitem{CR93} D. Calvetti and L. Reichel, {\sl Fast inversion of
    Vandermonde--like matrices involving orthogonal polynomials}, BIT, 1993.


\bibitem{EG99} Y.Eidelman, I.Gohberg, {\sl Linear complexity inversion algorithms for a class of structured matrices}, Integral Equations and Operator Theory (1999).

\bibitem{EGO05} Y.Eidelman, I.Gohberg and V. Olshevsky {\sl Eigenstructure of Order-One-Quasiseparabale Matrices. Three-term and Two-term Recurrence Relations}, Linear algebra and its applications, {\bf 405} (2005) 1-40.


\bibitem{GKO95} I. Gohberg, T. Kailath and V. Olshevsky, {\sl Fast Gaussian elimination with partial pivoting for matrices with displacement structure}, Math. Comput., 64-212(1995), 1557-1576.

\bibitem{GO94} I. Gohberg and V. Olshevsky,  {\sl Fast inversion of Chebychev-Vandermonde matrices},  Numerische
    Mathematik, {\bf 67, 1} (1994), 71-92.

\bibitem{GO942} I. Gohberg and V. Olshevsky,  {\sl Complexity of mutiplication with vectors for structured matrices},  Linear Algebra Appl., {\bf 202} (1994), 163-192.


\bibitem{GO97} I.Gohberg and V.Olshevsky. {\sl The fast generalized Parker-Traub algorithm for inversion of Vandermonde and related matrices}, J. of Complexity, {\bf 13, No.2} (1997), 208-234.

\bibitem{HR84} G. Heinig and K. Roast, {\sl Algebraic methods for Toeplitz-like matrices and operators}, Operator Theory, {\bf 13}, Birkauser Verlag, Basel, 1984.

\bibitem{KKM79} T. Kailath, S. Kung and M. Morf, {\sl Displacement ranks of matrices and linear equations}, J. Math. Anal. and Appl., {\bf 68} (1979), 395-407.

\bibitem{KO95} T. Kailath and V. Olshevsky, {\sl Displacement structure approach to Chebyshev-Vandermonde and related matrices}, Integral Equations and Operator Theory, {\bf 22} (1995), 65-92.


\bibitem{KO97} T. Kailath and V. Olshevsky, {\sl Displacement structure approach to polynomial Vandermonde and related matrices}, Linear Algebra and Application, {\bf 261} (1997), 49-90.

\bibitem{KS95} T. Kailath and A. H. Sayed, {\sl Displacement structure: theory and applications}, SIAM Review, {\bf 37, 3} (1995), 297-386.

\bibitem{MB79} J. Maroulas, S. Barnett, {\sl Polynomials with respect to a general basis I Theory}, Math. Analysis and Appl., {\bf 72} (1979) 177-194. 


\bibitem{O98} V.Olshevsky. {\sl Eigenvector computation for almost unitary Hessenberg matrices and inversion of Szego-Vandermonde matrices via Discrete Transmission lines}, Linear Algebra and Its Applications, {\bf 285} (1998), 37-67.

\bibitem{O01} V. Olshevsky, {\sl Associated polynomials, unitary Hessenberg
    matrices and fast generalized Parker-Traub and Bjorck-Pereyra algorithms
    for Szego-Vandermonde matrices} invited chapter in the book
    ``Structured Matrices: Recent Developments in Theory and
    Computation,'' 67-78, (D.Bini, E. Tyrtyshnikov, P. Yalamov.,
    Eds.), 2001, NOVA Science Publ., USA.



\bibitem{P64} F. Parker, {\sl Inverses of Vandermonde matrices}, Amer. Math.
    Monthly, {\bf 71} (1964), 410 - 411.


\bibitem{RO91} L.Reichel and G.Opfer, {\sl Chebyshev-Vandermonde systems}, Math. of Comp., {\bf 57} (1991), 703-721.

\bibitem{H90} N.J.Higham, {\sl Stability analysis of algorithms for solving confluent Vandermonde-like systems}, SIAM J. Matrix Anal. Appl., {\bf 11, 1} (1990), 23-41.

\bibitem{T66}  J. Traub, {\sl Associated polynomials and uniform methods for the
solution of linear problems}, 
   SIAM Review, {\bf 8, 3} (1966), 277-301.


\bibitem{V88} L. Verde-Star, {\sl Inverses of generalized Vandermonde matrices}, 
 J. Math. Anal. Appl., {\bf 131} (1988), 341-353.


\end{thebibliography}
\end{document}